\documentclass[oneside,reqno]{amsart}
\usepackage[bottom=0.5in]{geometry}
\usepackage{graphicx}%
\usepackage{multirow}%
\usepackage{amsmath,amssymb,amsfonts}%
\usepackage{amsthm}%
\usepackage{newtxmath}
\usepackage{fix-cm}
\usepackage[title]{appendix}%
\usepackage{xcolor}%
\usepackage{textcomp}%
\usepackage{manyfoot}%
\usepackage{booktabs}%
\usepackage{algorithm}%
\usepackage{algorithmicx}%
\usepackage{algpseudocode}%
\usepackage{listings}%
\usepackage{lmodern}
\usepackage{thmtools}
\usepackage{esint}

\usepackage{breqn}
\usepackage{bbm}
\usepackage{setspace}
\usepackage[T1]{fontenc}
\usepackage{BOONDOX-uprscr}
\usepackage[english]{babel}
\usepackage[utf8]{inputenc}
\usepackage{csquotes}
\usepackage{mathtools}
\usepackage{hyperref}
\hypersetup{
    colorlinks=true,
    linkcolor=black,
    filecolor=magenta,
    urlcolor=cyan,
}
\usepackage[backend=biber]{biblatex}
\AtBeginDocument{%
  \makeatletter
  \let\oldautoref\autoref
  \renewcommand{\autoref}[1]{\textit{\oldautoref{#1}}}%
  \makeatother
}

\allowdisplaybreaks
\DeclareMathAlphabet{\mathpzc}{OT1}{pzc}{m}{it}
\DeclarePairedDelimiter\floor{\lfloor}{\rfloor}

\newtheorem{lemma}{Lemma}
\newtheorem{thm}{\bf Theorem}

\newtheorem*{thm1.1}{\bf Theorem 1.1}
\newtheorem*{thm1.2}{\bf Theorem 1.2 (Unconditional)}
\newtheorem*{lem5.3}{\bf Lemma 5.3}

\theoremstyle{definition}

\theoremstyle{remark}
\newtheorem{remark}{ \bf Remark}
\newcommand{\dblprime}{^{\prime\prime}}
\numberwithin{equation}{section}
 \newcommand{\starsum}{\sideset{}{^\star}\sum}
 \newcommand{\lr}[1]{\left(#1\right)}
\newcommand{\ol}{\overline}
\newcommand{\GL}{\mathrm{GL}}
\newcommand{\SL}{\mathrm{SL}}

\begin{document}
\title[]{Sum of the $GL(3)$ Fourier coefficients over quadratics}

\author[1]{Himanshi Chanana}

\author[2]{Saurabh Kumar Singh}

\email{hchanana20@gmail.com,  skumar.bhu12@gmail.com}


\address{ Department of Mathematics and Statistics,
Indian Institute of Technology Kanpur, 
Kalyanpur, Kanpur Nagar-208016, India.} 






\begin{abstract}
Let $\mathcal{A}(n)$ denote the $(1,n)\text{-th}$ Fourier coefficient of a  $\SL(3, \mathbb{Z})$ Hecke eigenform or the ternary divisor function $d_3(n)$. Let $Q(x,y)$ be a symmetric positive definite quadratic form. This article establishes an asymptotic formula with a power-saving error term for the following sum
\begin{equation*}
    \sum_{1 \leqslant m \leqslant X} \sum_{1 \leqslant n\leqslant Y}  \mathcal{A}(Q(m,n)), 
\end{equation*}
where $X>1$ and $Y\leqslant X$.
\end{abstract}

\maketitle

\section{Introduction}
In analytic number theory, it is a classical problem to study the hidden structures underlying the Fourier coefficients of an automorphic form. An effective way is to consider its summatory function over a certain sequence. The order of magnitude of the individual values of the $n$-th Fourier coefficient fluctuates irregularly with $n$. However, the behavior of the corresponding summatory function is significantly more regular. The sums of this shape have been considered by several researchers in the last few decades. These studies aim to explore the underlying regularities and provide insights into the properties of arithmetic sequences. More precisely, they have considered the estimation of the sum
\begin{equation} \label{a(n)b(n)}
    \sum_{n \leqslant X} \mathsf{a}(n) \mathsf{b}(n),
\end{equation}
for various natural arithmetic sequences $\mathsf{a}(n)$ and $\mathsf{b}(n)$. For an interesting arithmetic function $\mathsf{a}(n)$, consider $\mathsf{b}(n)$ as a sparse sequence, such as a polynomial sequence. Even for linear polynomials the corresponding sum 
\begin{equation*}
    \sum_{\substack{n \leqslant X \\ n \equiv \alpha \textrm{ mod } q}} \mathsf{a}(n),
\end{equation*}
is quite difficult to estimate when the modulus $q$ is large relative to $X$, see for example \cite{FI}, where $\mathsf{a}(n)$ is the generalized divisor function. For $k \geqslant 2$, the generalized divisor functions, $d_k(n):= \# \{ (n_1,n_2,...,n_k) \in \mathbb{Z}_+^{k}: n=n_1 \cdots n_k\}$ have a long history and have recently played a significant role in the study of sums of arithmetic functions. In a sense, they are the most basic such functions and give the simplest model appearing in the context of $\GL(k)$ and having generating Dirichlet series $\zeta^k(s)$. For higher-degree polynomials, the problem becomes increasingly challenging and often requires very heavy machinery. Such issues arise also when considering, for example, sums of the following form
\begin{equation*}
    \sum_{n \leqslant X} d_k(n)d_k(n+\ell) \ \textrm{ or } \  \sum_{n \leqslant X} d_k(n^2 + n\ell),
\end{equation*}
which are useful in proving bounds for the moments of L-functions and in establishing non-vanishing results. Some of the classical results for the divisor function are \cite{Erdos, Ingham, Hooley1, Hooley2, Bykovskii}. For more general Fourier coefficients associated to $\GL(2)$ forms, see \cite{Blomer} and \cite{Templier}. The results of Blomer and Templier were further improved by Templier and Tsimerman in \cite{TT}. For coefficients associated with higher rank forms, Friedlander and Iwaniec were the first to study $d_3$ over $x^2+y^6$ \cite{FId3}. When we fix $\mathsf{a}(n)$ to be $d_k(n)$, the sums of the form \autoref{a(n)b(n)} arise as a tool in the study of the distribution of primes.

\medskip

\noindent This article deals with a problem of similar type; when $k=3$, $\mathsf{a}(n)$ is the $n$-th Fourier coefficient of general $\GL(3)$ form, and $\mathsf{b}(n)$ is the number of representations of the integer $n$ by an integral symmetric positive definite quadratic form $Q$ i.e., $n=Q(n_1,n_2)$ for $n_1, n_2 \in \mathbb{Z}$. This also improves and generalizes our previous work \cite{HC} in which we have proved a non-trivial bound for the mean of the Fourier coefficient of $\SL(3,\mathbb{Z})$ Hecke-Maass cusp form over the form $Q$. 

\smallskip
\noindent Precisely to understand the sum under consideration. Let $\phi$ be a Maass form of type $(\nu_1,\nu_2) \in \mathbb{C}^2$ for the group $\SL(3,\mathbb{Z})$, and let $A_\phi(m,n) \in \mathbb{C}$, denote the $(m,n)$-th normalized Fourier coefficient of $\phi$. Let $Q(x,y) = A x^2 + C xy + B y^2 \in \mathbb{Z}[x,y]$ be a symmetric positive definite quadratic form. In \cite{HC} we proved the following: for any $\varepsilon>0$ and $X>1$;
\begin{equation} \label{HC JNT}
    \sum_{n_1 \in \mathbb{Z}} \sum_{n_2 \in \mathbb{Z}} A_\phi(1,Q(n_1,n_2)) W_1\lr{\frac{n_1}{X}} W_2\lr{\frac{n_2}{X}} \ll X^{2-1/68+\varepsilon},
\end{equation}
where $W_1, W_2$ are smooth bump functions. The implied constant above depends on the form $\phi$ and $\varepsilon$. Let $\mathcal{A}(n)$ be either $A_\phi(1,n)$ or $d_3(n)$, and $Y \leqslant X$. Consider the sum
\begin{equation} \label{S(X,Y)}
    \mathcal{S}(X,Y):= \sum_{X \leqslant n_1\leqslant 2X} \sum_{Y \leqslant n_2 \leqslant 2Y} \mathcal{A}(Q(n_1,n_2)).
\end{equation}
As our primary result, we establish the following asymptotic formula with a power-saving error term.

\medskip

\begin{thm} \label{d3 thm}
    Let $X>1$, $1 \leqslant R < Y$, and $X^{3/4}R^2 < Y \leqslant X$. Let $W_i$, for $i=1,2$ be smooth functions supported on $[1,2]$ satisfying $x^k W_i^{(k)}(x) \ll R^k$, for all $k \geqslant 0$, and $Q(x,y)$ be a symmetric positive definite quadratic form over integers. Then for any $\varepsilon>0$, we have
    \begin{multline*}
        \sum_{n_1 \in \mathbb{Z}} \sum_{n_2 \in \mathbb{Z}} d_3(Q(n_1,n_2)) W_1\left( \frac{n_1}{X} \right) W_2\left( \frac{n_2}{Y} \right) \\ 
        = 2XY(\log{X})^2 \mathcal{C}_0 \mathcal{J}_0 + XY \log{X} (\mathcal{C}_1 \mathcal{J}_0 + 2 \mathcal{C}_0 \mathcal{J}_1) + \frac{1}{2}XY (\mathcal{C}_2\mathcal{J}_0 + \mathcal{C}_1\mathcal{J}_1 + \mathcal{C}_0 \mathcal{J}_2) \\ + \mathcal{O}_\varepsilon(X^{2-1/4+\varepsilon} R^2).
    \end{multline*}
    Here for $j=0,1,2$; $\mathcal{C}_j$'s are the singular series given by
    \begin{align*}
    \mathcal{C}_j &:= \sum_{q =1}^\infty \frac{1}{q^4} \sum_{n \mid q} n d(n) P_j(n,q)  \starsum_{a \textrm{ mod }q} \sum_{\alpha_1 \textrm{mod }q} \sum_{\alpha_2 \textrm{mod }q} e\left(\frac{-aQ(\alpha_1, \alpha_2)}{q}\right) S\lr{\ol{a},0;\frac{q}{n}} , 
\end{align*}
and $\mathcal{J}_j$'s are the singular integrals defined as follows
\begin{equation*}
    \mathcal{J}_j := \int_{\mathbb{R}} \int_{1/2}^{5/2} (\log{y})^j e\lr{xy} dy \iint_{\mathbb{R}^2} W_1(u) W_2(v) e\left(-xQ(u,v) \right) du dv dx.
\end{equation*}
\end{thm}

\medskip

\noindent Our next result is an immediate corollary to \autoref{d3 thm}, which gives us an asymptotic formula for the corresponding sharp cut-off. 

\medskip

\begin{thm} \label{sharp d3}
    Let $X>1$, $X^{3/4}<Y \leqslant X$, and $Q(x,y)$ be as defined in \autoref{d3 thm}. For any $\varepsilon>0$, we have
    \begin{align*}
        \sum_{X \leqslant n_1 \leqslant 2X} \sum_{Y \leqslant n_2 \leqslant 2Y} d_3(Q(n_1,n_2)) &= 2XY(\log{X})^2 \mathcal{C}_0 \mathcal{J}_0 + XY \log{X} (\mathcal{C}_1 \mathcal{J}_0 + 2 \mathcal{C}_0 \mathcal{J}_1) \\
        &+ \frac{1}{2}XY (\mathcal{C}_2\mathcal{J}_0 + \mathcal{C}_1\mathcal{J}_1 + \mathcal{C}_0 \mathcal{J}_2) + \mathcal{O}_\varepsilon(X^{1+1/4+\varepsilon}Y^{2/3}),
    \end{align*}
    where for $j=0,1,2$; $\mathcal{C}_j$'s and $\mathcal{J}_j$'s are defined as in \autoref{d3 thm}.
\end{thm}

\begin{remark}
    For the case of the divisor function, the analogous results as in \autoref{sharp d3} were proved by Gafurov \cite{Gafurov} for $Q(x,y)=x^2+y^2$ with an error term $\mathcal{O}(X^{5/3}\log^9{X})$. Subsequently, Yu \cite{Yu} refined this error term to $\mathcal{O}_{\varepsilon}(X^{3/2+\varepsilon})$. Recently, Dai \cite{Dai} generalizes Yu's result for $Q(x,y)=x^2+Ny^2$, with an error term of the same strength.
\end{remark}

\medskip

\begin{remark}
    An analogous asymptotic formula can be derived for the slightly modified sum
    \begin{equation*}
        \sum_{X \leqslant n_1^2 + n_2^2 \leqslant 2X} d_3(n_1^2 + n_2^2). 
    \end{equation*}
    This sum is simpler than the sum in \autoref{sharp d3} since
    \begin{equation*}
        \sum_{X \leqslant n_1^2 + n_2^2 \leqslant 2X} d_3(n_1^2 + n_2^2) = \sum_{X \leqslant n \leqslant 2X} d_3(n) r_2(n),
    \end{equation*}
    where $r_2(n) = \# \{ (n_1,n_2) \in \mathbb{Z}^2 : n_1^2 + n_2^2 =n \}$. 
\end{remark}

\medskip
\noindent The following two results are analogous to \autoref{d3 thm} and \autoref{sharp d3} in the case of Fourier coefficients of Hecke-Maass forms.

\medskip
\begin{thm} \label{GL3 FC}
    Let $X>1$ and $X^{3/4}R^2 < Y\leqslant X$. Let $A_\phi(1,n)$ denote the normalized $(1,n)$-th Fourier coefficient of a Hecke-Maass form $\phi$ for the group $\SL(3, \mathbb{Z})$, and $Q(x,y)$ be as defined in \autoref{d3 thm}. Then for any $\varepsilon>0$, we have 
\begin{equation*}
  \sum_{n_1 \in \mathbb{Z}} \sum_{n_2 \in \mathbb{Z}}  A_\phi(1,Q(n_1, n_2)) W_1\left( \frac{n_1}{X} \right) W_2\left( \frac{n_2}{Y} \right) \ll_{\phi, \varepsilon} X^{2-\frac{1}{4}+ \varepsilon}R^2,
\end{equation*}
where $W_1$ and $W_2$ are smooth bump functions supported on the  interval $[1, 2]$ as defined in \autoref{d3 thm}. 
\end{thm}

\medskip

\noindent Taking $R=1$ gives us an extra savings of $16/68$ over our previous result in \autoref{HC JNT}. For the sharp cut-off, we have the following bound.

\medskip

\begin{thm} \label{sharp FC}
    Let $A_\phi(1,n)$, and $ Q(x,y)$ be defined as in \autoref{GL3 FC}. Then for any $\varepsilon>0$, we have 
\begin{equation*}
  \sum_{X \leqslant n_1 \leqslant 2X} \sum_{Y \leqslant n_2 \leqslant 2Y}  A_\phi(1,Q(n_1, n_2)) \ll_{\phi, \varepsilon} X^{1+\frac{1}{4}+ \varepsilon} Y^{2/3},
\end{equation*} 
which is non-trivial provided $X^{3/4} < Y \leqslant X$.
\end{thm}

\begin{remark}
    We work with the coefficients $A_\phi(1,n)$ since they occur in the definition of the standard $L$-function 
$L(s,\phi) = \sum_{n\geq 1} A_\phi(1,n)n^{-s}$ attached to a $\SL(3,\mathbb{Z})$ Maass cusp form $\phi$. 
The coefficients $A_\phi(n,1)$, in contrast, correspond to those of the dual form $\tilde{\phi}$.
\end{remark}

\subsection*{\bf Idea behind the proof}
In this article, we apply the version of the `delta method' due to Duke, Friedlander, and Iwaniec given in  \autoref{circle method}. One of the advantages here is that this variant given in \autoref{delta} enables us to control the variation of the weight function in both $q$ (moduli) and $n$ variables, and they can be separated at negligible cost (see \cite{HC} for the Heath-Brown variant of the delta method). The weight function is now expressed as the Fourier inversion of the earlier form, allowing it to be conveniently represented in terms of the additive character $e(nu)$ with $|u|$ being small. After detecting $r=Q(n_1,n_2)$ using this delta method, the expression for $\mathcal{S}(X,Y)$ becomes
\begin{equation*}
    \mathcal{S}(X,Y) \approx \sum_{q \sim \mathcal{Q}} \ \starsum_{a \textrm{ mod } q} \sum_{r \sim X^2} \mathcal{A}(r) e\lr{\frac{ar}{q}} \sum_{n_1 \sim X} \sum_{n_2 \sim Y} e\lr{\frac{-aQ(n_1,n_2)}{q}},
\end{equation*}
where the $n_1$ and $n_2$ sums have analytic oscillations of size $R$, and $\mathcal{Q}=X$. Next, we apply the Voronoi summation formula to the $r$ sum, followed by the Poisson summation formula to $n_1$ and $n_2$ sums. In the $r$ sum, we save $X^2/\mathcal{Q}^{3/2}= X^{1/2}$. In the $n_1$ and $n_2$ sums, we save $X/\sqrt{\mathcal{Q}}R$ and $Y/\sqrt{\mathcal{Q}}R$, respectively. The additional $1/\sqrt{R}$ factor arises as we do not save from the integrals that come from applying Poisson. For details, see \autoref{voronoi} and \autoref{poisson section}. With this, our expression takes the following form.
\begin{equation} \label{M+E sketch}
    \mathcal{S}(X,Y) \approx \mathcal{S}_M(X,Y) +  \sum_{q \sim \mathcal{Q}} \ \sum_{r^* \sim X} \frac{\mathcal{B}(r^*)}{r^{* 1/3}} \sum_{n_1^* \sim R} \sum_{n_2^* \sim \frac{RX}{Y}} \mathfrak{C}(n_1^*,n_2^*;q),
\end{equation}
where the character sum is given by
\begin{equation*}
    \mathfrak{C}(n_1^*, n_2^*;q) = \starsum_{a \textrm{ mod } q} \sum_{\alpha_1 \textrm{ mod } q} \sum_{\alpha_2 \textrm{ mod } q} e\lr{\frac{-\ol{a}(Q(\alpha_1,\alpha_2) - n_1^*\alpha_1- n_2^* \alpha_2)}{q}} S(\ol{a},r^*;q),
\end{equation*}
$\mathcal{S}_M(X,Y)$ denotes the part of $\mathcal{S}(X,Y)$ which will contribute to the main term (when $n_1^*=n_2^*=0$) and $\mathcal{B}(n)$ is the dual Fourier coefficient. The $\mathcal{S}_M(X,Y)$ survives only in the case of \autoref{d3 thm} and \autoref{sharp d3}. For the exposition of the sketch, we will focus on the second term of \autoref{M+E sketch}; for details of the main term, see \autoref{main term section}. In the error part, additionally, we save $\sqrt{Q}$ from the sum over $a$, which gives us a total savings of $XY/R^2$. Hence, the bound now becomes $X^2 R^2$. To get the extra savings, earlier we were using the first derivative bound only to treat the integrals; in contrast, we have now done extensive analysis of integrals using the stationary phase method. At last, applying the partial summation formula, opening up the Kloosterman sum, and then using Miller's bound given in \autoref{additive twist} gives us the additional savings of $X^{1/4}$.

\medskip

\begin{remark}
    The usual procedure of applying Cauchy, then Poisson, will not work here, as we will not get enough savings in the off-diagonal case.
\end{remark}

\medskip

\noindent {\bf Notations.} Throughout the paper, the notation $a \ll A$ shall signify that, for any $\varepsilon> 0$ there exists a constant $c$ such that $|a| \leq cA X^\varepsilon$. The notation $ B \asymp C$ denotes that both $B \ll C$ and $C \ll B$ hold. Additionally, $D \sim E$, implies that $E\leq D \leq 2E$. The notation $ a \mid b^\infty$ means the prime factors of $ a$ are a subset of the prime factors of $ b$ for $a,b \in \mathbb{N}$. The symbol $\varepsilon$ represents a suitably small positive quantity that may vary at different instances and $e(x) = e^{2\pi i x}$. \\

\medskip
\noindent To establish our results, we first recall essential definitions and results, including the Voronoi summation formula and properties of Fourier coefficients.

\section{Preliminaries}
\subsection{\texorpdfstring{$\mathbf{\GL(3)}$}{} Voronoi summation} Let $\phi$ be a Maass form of type $\nu = (\nu_1, \nu_2) \in \mathbb{C}^2$ for the group $\SL(3,\mathbb{Z})$ such that $\phi$ is an eigenfunction of all the Hecke operators $T_n \ (n \in \mathbb{N})$ with eigenvalues $A_\phi(n,1) \in \mathbb{C}$, normalized so that $A_\phi(1,1) = 1$. \\

\noindent Let the Fourier-Whittaker expansion of $\phi(z)$ be given by 
\begin{align} \label{whittaker}
    \phi(z) &= \sum_{\gamma \in U_2(\mathbb{Z}) \symbol{92} \SL(2,\mathbb{Z}) } \;\; \sum_{m=1}^\infty \;\sum_{n \neq 0} \frac{A_\phi(m,n)}{|m n|} \notag\\ 
     &\hspace{2cm}\times \textbf{W}_{\text{Jacquet}} \left( \begin{pmatrix}
        |m n| & & \\ & m & \\ & & 1
    \end{pmatrix} \begin{pmatrix}
        \gamma & \\ & 1
    \end{pmatrix}z, \nu, \psi_{1, \frac{n}{|n|} } \right),
\end{align}
where $\textbf{W}$ is the Whittaker-Jacquet function (for more details, see Goldfeld's book \cite{Goldfeld}). We introduce the  Langlands parameters $(\alpha_1, \alpha_2, \alpha_3)$, which are defined by 
\begin{equation} \label{langland}
    \alpha_1= - \nu_1 -2\nu_2 + 1, \; \alpha_2 = -\nu_1 + \nu_2 \; \text{ and } \alpha_3 = 2 \nu_1 + \nu_2 - 1. 
\end{equation}
The Ramanujan-Selberg conjecture predicts that $|\text{Re}(\alpha_j)| = 0$, and, according to the work of Jacquet-Shalika, we know that $| \text{Re}( \alpha_j)| < \frac{1}{2}$. The following lemma provides the Ramanujan bound for $A_\phi(m,n)$ on average.

\medskip
\begin{lemma} \label{Ramanujan bound}
    Let $A_\phi(m,n)$ be as given in \autoref{whittaker}. We have
    \begin{equation}
        \mathop{\sum\sum}_{m^2 n \ll X} \; |A_\phi(m,n)|^2 \ll_\phi X^{1 + \varepsilon},
    \end{equation}
    where the implied constant depends on the form $\phi$ and $\varepsilon$.
\end{lemma}
\begin{proof}
    For the proof refer to \cite{Molteni}.
\end{proof}

\noindent The next lemma will be used quite often in the analysis.

\smallskip
\begin{lemma}\label{RB GL3}
We have
\begin{align*}
\mathop{\sum}_{ n \leqslant X} \vert A_\phi(m, n)\vert ^{2} \ll_{\phi} \, m^{2\theta + \epsilon}\, X^{1 + \epsilon}, 
\end{align*} where $ \theta $ is bound towards the Ramanujan conjecture. 
\end{lemma}
\begin{proof}
From the Hecke relation, we have
\begin{align*}
A_\phi(m, n) = \sum_{d | (m, n)}\,\mu(d) A_\phi\left(\frac{m}{d}, 1\right) A_\phi\left(1, \frac{n}{d}\right), 
\end{align*}
where $\mu(n)$ is the M\"obius function. By applying the Cauchy-Schwarz inequality, we obtain
\begin{align*}
     | A_\phi(m, n) |^2 & \leqslant \sum_{d | (m, n)} \left| \mu(d) A_\phi \left(\frac{m}{d}, 1\right)\right|^2  \sum_{d | (m, n)} \left| A_\phi\left(1, \frac{n}{d}\right) \right|^2  \\
     & \leqslant  \sum_{d | m } \left|  A_\phi\left(\frac{m}{d}, 1\right)\right|^2 \sum_{d |  n} \left| A_\phi\left(1, \frac{n}{d}\right) \right|^2. 
\end{align*}
Using the individual bound $A_\phi(m, n) \ll_{\epsilon} (mn)^{\theta + \epsilon}$ for $A_\phi({m}/{d}, 1)$, where $\theta \leq 5/14$ is the bound towards the Ramanujan conjecture on $\GL(3)$ (see \cite{Kim}), we get

\begin{align*}
\mathop{\sum}_{ n \leqslant X} \vert A_\phi(m, n)\vert ^{2} & \leqslant \mathop{\sum}_{ n \leqslant X}  \sum_{d | m } \left|  A_\phi\left(\frac{m}{d}, 1\right)\right|^2 \sum_{d |  n} \left| A_\phi\left(1, \frac{n}{d}\right) \right|^2 \\
& \ll m^{2 \theta } \sum_{d | m } \frac{1}{d^{2 \theta}} \mathop{\sum}_{ n \leqslant X}  \sum_{d |  n} \left| A_\phi\left(1, \frac{n}{d}\right) \right|^2 . 
\end{align*}
Replacing $ n$ by $ n_1 d$ in the last sum and using \autoref{Ramanujan bound} for the $n$-sum, we obtain
\begin{align*}
    \mathop{\sum}_{ n \leqslant X} \vert A_\phi(m, n)\vert ^{2} &  \ll m^{2 \theta + \epsilon  } \sum_{d \leqslant X} \mathop{\sum}_{ n_1 \leqslant X/d} \left| A_\phi\left(1, n_1 \right) \right|^2  \\ 
    & \ll_{\phi} \ m^{2 \theta + \epsilon  } \sum_{d \leqslant X} \frac{ X^{1 + \epsilon}}{ d}  \ll_{\phi} \  m^{2 \theta + \epsilon  } X^{1 + \epsilon}.
\end{align*}

\end{proof}

\noindent The following non-trivial bound for the additive twist of the Fourier coefficients was proved by Miller.

\medskip
\begin{lemma} \label{additive twist}
 
 Let $\phi$ be a Hecke eigenform on the group  $ \SL(3, \mathbb{Z}) $ with the Fourier coefficients $A_\phi(m, n)$.  For any $\alpha \in \mathbb{R}$, we have 
 
 \begin{equation}
  \sum_{n \leqslant X} A_\phi(m,n) e(n \alpha ) \ll X^{3/4 + \varepsilon}. 
 \end{equation}
 The implied constant depends only on $m$, $\varepsilon$ and $\phi$, and is independent of $\alpha$. 
  
 \end{lemma}
 \begin{proof}
  See Theorem $1.1$  of \cite{Miller}. 
 \end{proof}
 
\medskip
\noindent We now recall the Voronoi summation formula for $\GL(3)$ Hecke eigenvalues (refer \cite{MS}) and $d_3$ (refer \cite{XLI}), which will play a crucial role in our analysis. Let $g$ be a compactly supported function on $(0, \infty)$. The Mellin transform of $g$ is defined by $\Tilde{g}(s) = \int_{0}^{\infty} g(x) x^{s-1} dx$, where $s = \sigma + i t$. For $\sigma > -1 + \text{max} \{ -\text{Re}(\alpha_1), - \text{Re}(\alpha_2), - \text{Re}(\alpha_3)\} $ and $\ell = 0, 1$, we define 
\begin{equation} \label{G}
    G_\ell(y) = \frac{1}{2\pi i} \int_{(\sigma)} (\pi^3 y)^{-s} \prod_{j=1}^3 \frac{\Gamma\left(\frac{1+s+\alpha_j + \ell}{2}\right)}{\Gamma\left(\frac{-s-\alpha_j + \ell}{2}\right)} \Tilde{g}(-s) \;ds,
\end{equation}
where $\alpha_j$'s are Langlands parameters as specified in \autoref{langland}. We set
\begin{equation} \label{Gpm}
    G_{\pm}(y) = \frac{1}{2 \pi^{3/2}} \bigl(G_{0}(y) \mp i G_1(y)\bigr).
\end{equation}
The Kloosterman sum is defined by 
\begin{equation*}
    S(a,b;q) = \sideset{}{^\star}\sum_{x \textrm{ mod } \;q} e\left(\frac{ax + b \overline{x}}{q} \right),
\end{equation*}
where $\overline{x}$ denotes multiplicative inverse of $x$ modulo $q$. We now record the Voronoi summation formula for $\GL(3)$ forms.

\medskip

\begin{lemma} \label{voronoigl3}
    Let $g \in C_c^{\infty}(0, \infty)$. Let $A_\phi(m,n)$ be the $(m,n)$-th Fourier coefficient of a Maass form $\phi$ for $\SL(3,\mathbb{Z})$, we have
    \begin{multline}
        \sum_{n=1}^{\infty} A_\phi(m,n) e\left(\frac{an}{q}\right) g(n) = q\sum_{\pm} \sum_{n_1 \mid qm} \sum_{n_2 = 1}^{\infty} \frac{A_\phi(n_1,n_2)}{n_1 n_2} S\left(m\overline{a},\pm n_2; \frac{mq}{n_1}\right) \\ \times G_{\pm}\left(\frac{n_1^2n_2}{q^3 m}\right),
    \end{multline}
    where $(a,q) = 1$ and $a\overline{a} \equiv 1 \textrm{ mod }  q$.
\end{lemma}
\subsection{ \bf Voronoi summation for the ternary divisor function} The Voronoi formula for $d_3$ was proved by Ivi\'{c}, and later Li \cite{XLI} came up with a more explicit formula for it. To achieve his result, set 
\begin{equation} \label{sigma 00}
    \sigma_{0,0}(m,n) = \sum_{\substack{d\mid n \\ d > 0}} \sum_{\substack{d^\prime \mid \frac{n}{d}\\ d^\prime>0 \\ (d^\prime,m) = 1}} 1 = \sum_{e \mid(m,n)} \mu(e)d_3(n/e).
\end{equation}
For $h \in C_c^{\infty}(0,\infty),$ for $\ell=0,1$ and $\sigma > -1-2\ell$ , set
\begin{equation*}
    H_\ell(y) = \frac{1}{2 \pi i} \int_{(\sigma)} (\pi^3 y)^{-s} \frac{\Gamma(\frac{1 + s +2\ell}{2})^3}{\Gamma(\frac{-s}{2})^3} \tilde{h}(-s-\ell) ds,
\end{equation*}
and 
\begin{equation} \label{Hpm}
    H_\pm (y) = \frac{1}{2 \pi^{3/2}} \left(H_0(y) \mp \frac{i}{\pi^3 y} H_1(y)\right). 
\end{equation}
Observe that the behavior of $G_{\pm}$ is the same as $H_{\pm}$. Now, with the aid of the above terminology, we state the Voronoi summation formula for $d_3$ in the following lemma.

\medskip

\begin{lemma} \label{voronoi d3}
    Let $h \in C_c^{\infty}(0,\infty)$, $a, \overline{a}, q \in \mathbb{Z}^+$ with $a\overline{a} \equiv 1 \; \textrm{mod}\; q $ we have
\begin{dmath*}
\begin{aligned}
    \sum_{n \geqslant1}& d_3(n) e\left( \frac{a n}{q}\right) h(n) \\
   &= q \sum_{\pm}\sum_{n_1 \mid q} \sum_{n_2 =1}^\infty \frac{1}{n_1 n_2} \sum_{m_1 \mid n_1} \sum_{m_2 \mid \frac{n_1}{m_1}} \sigma_{0,0} \left( \frac{n_1}{m_1 m_2}, n_2 \right) S\left( \overline{a}, \pm n_2 ; \frac{q}{n_1}\right) H_\pm\left(\frac{n_1^2 n_2}{q^3}\right) \\
    & \hspace{1cm}+ \frac{1}{2q^2} \Tilde{h}(1) \sum_{n_1 \mid q} n_1 d(n_1)\; P_2(n_1,q)\; S\left( \overline{a}, 0 ; \frac{q}{n_1}\right) \\
    &\hspace{1cm}+ \frac{1}{2q^2} \Tilde{h}^\prime(1) \sum_{n_1 \mid q} n_1 d(n_1) \; P_1(n_1,q) \; S\left( \overline{a}, 0 ; \frac{q}{n_1}\right) \\
    &\hspace{1cm}+ \frac{1}{2q^2} \Tilde{h}\dblprime(1) \sum_{n_1 \mid q} n_1 d(n_1) \; P_0(n,q) \ S\left( \overline{a}, 0 ; \frac{q}{n_1}\right),
\end{aligned}
\end{dmath*}
where
\begin{equation} \label{P1}
   P_0(n_1,q)=1/2, \ \ \  P_1(n_1,q) = \frac{5}{3} \log{n_1} - 3 \log{q} + 3\gamma - \frac{1}{3 d(n_1)} \sum_{l \mid n_1} \log{l},
\end{equation}
and
\begin{multline} \label{P2}
    P_2(n_1,q) = (\log{n_1})^2 - 5\log{q} \log{n_1} + \frac{9}{2}(\log{q})^2 + 3\gamma^2 - 3\gamma_1 + 7\gamma \log{n_1} -9\gamma \log{q} \\
    + \frac{1}{d(n_1)} \left( (\log{n_1} + \log{q} - 5\gamma) \sum_{l \mid n_1} \log{l} -\frac{3}{2} \sum_{l \mid n_1}(\log{l})^2 \right),
\end{multline}
where $\gamma:= \lim_{s \rightarrow 1} \left(\zeta(s) - \frac{1}{s-1}\right)$ is the Euler constant and $\gamma_1:= -\frac{d}{ds} \left(\zeta(s) - \frac{1}{s-1}\right) \bigg|_{s=1}$ is the Stieltjes constant.
\end{lemma}

\medskip

\noindent The ternary divisor function satisfies the bound $d_3(n) \ll n^\varepsilon$ and bound analogous to \autoref{additive twist} (see \cite{Miller}). By \autoref{sigma 00}, trivially, we have
\begin{equation} \label{RB d3}
    \sum_{m \ll X} \Bigg| \sum_{m_1 \mid n} \sum_{m_2 \mid \frac{n}{m_1}} \sigma_{0,0}\left(\frac{n}{m_1 m_2}, m\right) \Bigg|^2 \ll \sum_{m \ll X} \left(d_3(n) d_3(m)\right)^2 \ll X^{1 + \varepsilon}.
\end{equation}
In the practical application of the Voronoi summation formula, it is essential to know the asymptotic behavior of $H_i$, for $i=0,1$. This requirement is fulfilled by the following lemma by X. Li \cite{XLI2}. 

\medskip

\begin{lemma} \label{GO}
 Suppose $h$ is a smooth function compactly supported on $[Y, 2Y]$, $H_\pm$ be as in \autoref{Hpm}. Then for any fixed integer $K \geq 1$ and $yY \gg 1$, we have
 \begin{equation*} 
     H_\pm(y) =  \pi^4 y \int_{0}^\infty h(z) \sum_{j=1}^K \frac{c_j e{(3 (yz)^{1/3})} + d_j e{(3 (yz)^{1/3})} }{(\pi^3 yz)^{j/3}} dz + \mathcal{O}\left((yY)^{\frac{-K + 2}{3}}\right),
 \end{equation*}
where $c_j$ and $d_j$ are absolute constants depending on $\alpha_1, \alpha_2$ and $\alpha_3$. In particular $c_1 = -d_1 =-\frac{2}{\sqrt{3\pi}}$.
\end{lemma}
\begin{proof}
    See \cite[Lemma 2.1]{XLI2} for the proof.
\end{proof}

\medskip

\noindent Having outlined the Voronoi summation formula for $\GL(3)$ coefficients, we now turn to the Poisson summation formula, another crucial tool in our analysis.
\subsection{\bf Poisson summation formula} Let $\phi : \mathbb{R}^n \rightarrow \mathbb{C}$ be any Schwartz class function. The Fourier transform of $\phi$ is defined as 
\begin{equation*}
    \hat{\phi}(\mathbf{y}) = \int_{\mathbb{R}^n} \phi(\mathbf{x}) e(- \mathbf{x.y}) \; d\mathbf{x},
\end{equation*}
where $d\mathbf{x}$ is the usual Lebesgue measure on $\mathbb{R}^n$. In the following lemma, we state the Poisson summation formula.

\medskip

\begin{lemma} \label{poisson formula}
    Let $\phi$ and $\hat{\phi}$ be as defined above. Then we have
    \begin{equation}
        \sum_{n \in \mathbb{Z}^n} \phi(n) = \sum_{m \in \mathbb{Z}^n} \hat{\phi}(m).
    \end{equation}
\end{lemma}
\begin{proof}
    See \cite[Theorem 4.4]{IK}.
\end{proof}
\subsection{Gauss Sum} Let $Q(\bf{x})$ be a positive definite quadratic form in $\bf{x} = (x_1,x_2,...,x_r)$. We define the Gauss sum associated with a quadratic form $Q(\mathbf{x})$ as
\begin{equation} \label{Gauss sum}
    G_{\mathbf{m}}\left(\frac{a}{q}\right) := \sum_{\mathbf{x} \; \textrm{mod } q} e\left(\frac{a}{q} \left(Q(\mathbf{x}) + \mathbf{m}^t \mathbf{x}\right) \right).
\end{equation}
We have the following precise expression for the Gauss sum defined above.

\medskip

\begin{lemma} \label{gauss sum}
    Let $(q,2|\mathbf{A}| a) = 1$, where $\mathbf{A}$ is matrix associated with quadratic form $Q$ and $|\mathbf{A}|$ denotes the determinant of $\mathbf{A}$ and $\mathbf{m} \in \mathbb{Z}^r$. We have
    \begin{equation*} 
         G_{\mathbf{m}}\left(\frac{a}{q}\right) = \left(\frac{|\mathbf{A}|}{q}\right) \left( \varepsilon_q \left( \frac{2 a}{q}\right) \sqrt{q} \right)^r e\left( -\frac{a}{q} Q^*(\mathbf{m}) \right),
    \end{equation*}
where $Q^*(\mathbf{x})$ is adjoint quadratic form, $N$ is an integer such that $N Q^*(\mathbf{x})$ has integral coefficients, $\left(\tfrac{\cdot}{q}\right)$ denotes the Kronecker symbol and 
\begin{equation} \label{varepsilon}
   \varepsilon_q = \begin{cases}
 1 & \text{ if }  q \equiv 1 \; \textrm{mod} \; 4 \\
i & \text{ if } q \equiv -1 \; \textrm{mod} \; 4
\end{cases}.
\end{equation}    
\end{lemma}
\begin{proof}
    See (\cite{IK}, page $475$, Lemma $20.13$).
\end{proof}

\noindent For a particular case when $x \in \mathbb{Z}$ and $ m=0$ in \autoref{Gauss sum}, we have the following result.

\medskip

\begin{lemma} \label{quadratic gauss sum}
If $(a,q) = 1$, we have
\begin{equation*}
    \sum_{x \; \textrm{mod } q} e\left( \frac{ax^2}{q}\right) = \begin{cases}
        0 & \text{if} \; q \equiv 2 \; \textrm{mod} \; 4 \\ \vspace{0.1cm}
        \varepsilon_q \sqrt{q} \left(\frac{a}{q} \right) & \text{if} \; q \equiv \pm 1 \; \textrm{mod} \; 4 \\ \vspace{0.1cm}
         \varepsilon_{q_0} \sqrt{q} \left( \frac{a}{q_0}\right) \lr{1 + e\lr{\frac{aq_0}{4}}} & \text{if} \; q =2^k q_0 \textrm{ with } k \equiv 0 \textrm{ mod } 2  \\ \vspace{0.1cm} 
        \varepsilon_{q_0} \sqrt{2q} \lr{\frac{a}{q_0}} e\lr{\frac{aq_0}{8}}  & \text{if} \; q =2^k q_0 \textrm{ with } k \equiv 1 \textrm{ mod } 2 
    \end{cases}
\end{equation*}
where $\varepsilon_q$ is defined in  \autoref{varepsilon}, $q_0$ is odd and $k\geq 2$.
\end{lemma}
\begin{proof}
    See (\cite{Khan+Young}, Lemma $5.1$).
\end{proof}
\subsection{Oscillatory integrals} We also recall the following estimates of exponential integrals. Let  $ g$ be a continuous function defined on an interval $[a, b] $.  Variation of the function $ g$ is defined by 
\begin{align} \label{variation of g}
\textrm{Var}_{[a,b]} g := \int_a^b |g^\prime(x)| dx .
\end{align}
Let
\[
I= \int_a^b g(x) e(F(x))dx. 
\]
\begin{lemma} \label{exponential sum}
{\bf (Exponential sum lemma)} Let $F$ be real and twice differentiable function on $[a,b]$, $g$ and $I$ are as above. Then if $F^\prime$ is monotone and $|F^\prime(x)|  \geqslant \mu_1 >0 $ for $a\leqslant x \leqslant b$, we have $I \ll \textrm{Var}_{[a,b]} g/\mu_1$. Further, if $|F\dblprime(x)| \geqslant\mu_2 >0$. Then we have $I \ll \textrm{Var}_{[a,b]} g/\mu_2^{1/2}$, where $\textrm{Var}_{[a,b]} g$ is defined in \autoref{variation of g}.

\end{lemma} 
\begin{proof}
See \cite[Lemma 2.1]{Hux}. 
\end{proof}

\noindent More generally, let $\mathcal{F}$ be an index set and $X=X_J: \mathcal{F} \rightarrow \mathbb{R}_{\geq1}$ be a function of $J \in \mathcal{F}$. A family of $\{u_J\}_{J \in \mathcal{F}}$ of smooth functions supported on a product of dyadic intervals in $\mathbb{R}_{>0}^d$ is called $X$-inert if for each $j=(j_1,...,j_d)\in \mathbb{Z}_{\geq0}^d$ we have
\begin{equation*}
    \sup_{J \in \mathcal{F}} \sup_{(x_1,...,x_d) \in \mathbb{R}_{>0}^d} X^{-j_1-...-j_d}_{J} \mid x_1^{j_1}...x_d^{j_d}u^{(j_1,...,j_d)}_J(x_1,...,x_d) \mid  \ \ll_{j_1,...,j_d} 1.
\end{equation*}
We will use the following stationary phase lemma several times.

\medskip

\begin{lemma}\label{stationary phase}
    Suppose $u_J$ is $X$-inert in $t_1,...,t_d$, supported on $t_1\asymp T$ and $t_i \asymp T_i$ for $i=2,...,d$. Suppose that on the support of $u_J$, $h=h_J$ satisfies that
    \begin{equation*}
        \frac{\partial^{a_1+a_2+...+a_d}}{\partial t_1^{a_1}...\partial t_d^{a_d}}h(t_1,...,t_d) \ll_{a_1,...,a_d} \frac{Y}{T^{a_1}}\frac{1}{T_2^{a_2}...T_d^{a_d}},
    \end{equation*}
    for all $a_1,...,a_d \in \mathbb{N}$. Let
    \begin{equation*}
        \mathcal{I} = \int_{\mathbb{R}} u_J(t_1,...,t_d)e^{i h(t_1,...,t_d)} dt_1.
    \end{equation*}
    \begin{itemize}
        \item[(a)] Suppose $\frac{\partial}{\partial t_1} h(t_1,...,t_d) \gg \frac{Y}{T}$ for all $(t_1,...,t_d) \in supp \ u_J$. Suppose $Y/X\geqslant1$. Then $\mathcal{I} \ll_C T(Y/X)^{-C}$ for $C$ arbitrary large.
        \item[(b)] Suppose $\frac{\partial^2}{\partial t_1^2} h(t_1,...,t_d) \gg \frac{Y}{T^2}$ for all $(t_1,...,t_d) \in supp \ u_J$, and there exists $t_0 \in \mathbb{R}$ such that $\frac{\partial}{\partial t_1}h(t_0,t_2,...,t_d)=0$. Suppose that $Y/X^2 \geqslant R \geqslant 1$. Then
        \begin{equation*}
            \mathcal{I} = \frac{T}{\sqrt{Y}} e^{ih(t_0,t_2,...,t_d)}U_J(t_2,...,t_d) + \mathcal{O}_C(TR^{-C}),
        \end{equation*}
        for some $X$-inert family of functions $U_J$ and $C>0$ may be taken to be arbitrarily large.
    \end{itemize}
\end{lemma}
\begin{proof}
    See Blomer-Khan-Young (\cite{BYK}, Section 8) and Kiral-Petrow-Young (\cite{KPY}, Section 3).
\end{proof}
\section{The Circle Method} \label{circle method}
We will begin this part with a well-known Fourier expansion of the $\delta$-symbol, which was developed by Duke, Friedlander, and Iwaniec and is presented in Chapter $20$ of \cite{IK}. More specifically, we will use the expansion $(20.157)$. 

\medskip

\begin{lemma} \label{dfi}
Let $\delta : \mathbb{Z} \rightarrow \{0,1\}$ be defined by
 \begin{equation} 
  \delta(x)= 
  \begin{cases}
   1 \ \ \ \ \ \textrm{if} \ \ \  x = 0, \\
   0 \ \ \ \ \ \textrm{if} \ \ \ x \neq 0.
  \end{cases}
 \end{equation}
Then for $n,m \in \mathbb{Z} \cap [-2X,2X],$ we have
 \begin{equation} \label{delta}
     \delta(n-m) = \frac{1}{\mathcal{Q}} \sum_{q =1}^\infty \frac{1}{q}\;
     \sideset{}{^\star}\sum_{a \;\textrm{mod} \;q} e\left(\frac{(n-m)a}{q}\right) \int_{\mathbb{R}} \psi(q,x)e\left(\frac{(n-m)x}{q\mathcal{Q}}\right)dx,
 \end{equation}
where $\mathcal{Q}=2 X^{1/2}$. The function $\psi$ in \autoref{delta} is not explicitly given. Nevertheless, the following properties of the function $\psi(q,x)$ are of our interest:
 \begin{align}
     &\psi(q,x) = 1 + h(q,x), \;\;\; \text{with} \;\;\; h(q,x) = \mathcal{O}\left(\frac{\mathcal{Q}}{q} \left(\frac{q}{\mathcal{Q}} + |x| \right)^A \right), \label{delta1}\\
     &\psi(q,x) \ll |x|^{-A}, \label{delta2} \\
     &x^j\frac{\partial^j}{\partial x^j} \psi(q,x) \ll \text{min} \left\{ \frac{\mathcal{Q}}{q}, \frac{1}{|x|}\right\} \log \mathcal{Q}, \;\; \text{for any}\;\; A>1, \;j \geqslant1. \label{delta3}
 \end{align}
 In particular, \autoref{delta2} implies that the effective range of integral in \autoref{delta} is $[-X^\epsilon, X^\epsilon].$ It also follows from \autoref{delta1} that if $q \ll \mathcal{Q}^{1-\epsilon}$ and $x \ll \mathcal{Q}^{-\epsilon}$, then $\psi(q,x)$ can be replaced by 1, at the cost of a negligible error term. If $q \gg \mathcal{Q}^{1-\epsilon}$, then we get $x^j\frac{\partial^j}{\partial x^j} \psi(q,x) \ll \mathcal{Q}^{\epsilon},$ for any $j \geqslant1$. If $q \ll \mathcal{Q}^{1- \epsilon}$ and $\mathcal{Q}^{-\epsilon} \ll |x| \ll \mathcal{Q}^{\epsilon}$, then $x^j\frac{\partial^j}{\partial x^j} \psi(q,x) \ll \mathcal{Q}^{\epsilon},$ for any $j \geqslant1$. Finally, by Parseval and Cauchy, we get
  \begin{equation*}
      \int_{\mathbb{R}} |\psi(q,x)| \; + \; |\psi(q,x)|^2 \; dx \ll \mathcal{Q}^\varepsilon,
  \end{equation*}
  i.e., $\psi(q,x)$ has average size one in the $L^1$ and $L^2$ sense.
 \end{lemma}
\section{Proof of \autoref{sharp d3} and \autoref{sharp FC}}
The sum defined in \autoref{S(X,Y)} can be expressed as follows:
\begin{equation} \label{S(X,Y) to S^*(X,Y)}
	\mathcal{S}(X,Y) = \mathop{\sum\sum}_{n_1,n_2 \in \mathbb{Z}} \mathcal{A}(Q(n_1,n_2))W_1 \left(\frac{n_1}{X}\right)W_2\left(\frac{n_2}{Y}\right) + \mathcal{O}_\varepsilon\lr{\frac{X^{1+\varepsilon}Y}{R}},
\end{equation}
where $W_i$ for $i=1,2$ are two smooth bump functions defined as in \autoref{d3 thm} with $W_i =1$ on $[1+R^{-1},2-R^{-1}]$. To get \autoref{S(X,Y) to S^*(X,Y)} for the Fourier coefficient $A_\phi(1,n)$, we have used \autoref{Ramanujan bound}. Now, using the results of  \autoref{d3 thm} and \autoref{GL3 FC}, we obtain $R=X^{-1/4}Y^{1/3}$. Thus, we get our desired estimates.

\medskip

\noindent We now proceed to prove \autoref{d3 thm} and \autoref{GL3 FC} by using the delta method.
\section{Proof of \autoref{d3 thm} and \autoref{GL3 FC}}
Let us denote the first term on the right side of \autoref{S(X,Y) to S^*(X,Y)} by $\mathcal{S}^*(X,Y)$. Now, we introduce the $\delta$-symbol and rewrite the expression for $\mathcal{S}^*(X,Y)$ as
\begin{equation*}
	\mathcal{S}^*(X,Y) = \mathop{\sum\sum}_{n_1,n_2 \in \mathbb{Z}} \sum_{r \in \mathbb{Z}} \mathcal{A}(r) \delta(r-Q(n_1, n_2)) V\left( \frac{r}{X^2} \right)W_1 \left(\frac{n_1}{X}\right)W_2\left(\frac{n_2}{Y}\right),
\end{equation*}
where $V$ is another smooth bump function supported on the interval $[1/2,5/2]$ with $V \equiv 1$ on $[1,2]$ and $x^jV^{(j)}(x) \ll_j 1$ for all $j \geqslant 0$. We then apply the expansion for $\delta$-symbol given in \autoref{dfi}, $\autoref{delta}$ to get the following expression. 
\begin{multline}	\label{SXY after CM}
	\mathcal{S}^*(X,Y) = \frac{1}{\mathcal{Q}}  \sum_{1 \leqslant q \leqslant \mathcal{Q}} \frac{1}{q}\ \sideset{}{^\star}\sum_{a \textrm{ mod }q} \int_{\mathbb{R}} \psi(q, x) U(x) \left[ \sum_{r \in \mathbb{Z}} \mathcal{A}(r)e\left(\frac{ar}{q}\right)V\left(\frac{r}{X^2}\right) e\left(\frac{rx}{q\mathcal{Q}}\right) \right] \\ \times \mathop{\sum\sum}_{n_1, n_2 \in \mathbb{Z}}  W_1 \left(\frac{n_1}{X}\right)W_2\left(\frac{n_2}{Y}\right)e\left(\frac{-a Q(n_1, n_2)}{q}\right)  e\left(\frac{-xQ(n_1,n_2)}{q\mathcal{Q}}\right) dx,  
\end{multline}
where $U(x)$ is also a smooth bump function supported on $[-2X^\varepsilon, 2X^\varepsilon]$ with $U \equiv 1$ on $[-X^\varepsilon, X^\varepsilon]$ and $x^j U^{(j)}(x) \ll_j 1$ for all $j \geq 0$.
\subsection{Voronoi summation formula} \label{voronoi}
In this subsection, we apply the Voronoi summation formula on the $r$-sum enclosed within the square bracket. Using \autoref{voronoigl3} if $\mathcal{A}(r)$ is Fourier coefficient of $\SL(3,\mathbb{Z})$ Hecke-Maass cusp form or \autoref{voronoi d3} for $d_3(n)$ with $v_x(y)= V\lr{\frac{y}{X^2}} e\lr{\frac{yx}{q\mathcal{Q}}}$. We have,
\begin{multline} \label{voronoi1}
   \sum_{r \in \mathbb{Z}} \mathcal{A}(r)e\left(\frac{ar}{q}\right)v_x(r) =  \frac{\mathsf{c}}{2q^2} \Tilde{\mathcal{V}}(v_x,a,q) \\ + q \sum_{\pm} \sum_{n\mid q} \sum_{m = 1}^\infty \frac{B(n,m)}{nm} S\left(\overline{a}, \pm m; \frac{q}{n}\right) V_{\pm}\left(\frac{n^2 m}{q^3}\right),
\end{multline}
where 
\begin{equation} \label{B(n,m)}
  B(n,m) = \Biggl\{ \begin{array}{ll}
    A_\phi(n,m) & \quad \text{if } \mathcal{A}(r) = A_\phi(1,r) \\ \vspace{0.2cm}
   \sum_{m_1 \mid n}\sum_{m_2 \mid \frac{n}{m_1}} \sigma_{0,0} \left(\frac{n}{m_1m_2}, m\right) & \quad \text{if } \mathcal{A}(r) = d_3(r) 
      
 \end{array},\Biggr.
\end{equation}
\begin{multline} \label{V tilde}
    \Tilde{\mathcal{V}}(v_x,a,q) = \Tilde{v_x}(1) \sum_{n \mid q} n d(n) P_2(n,q)\; S\left( \overline{a}, 0 ; \frac{q}{n}\right) \\
    + \Tilde{v_x}^\prime(1) \sum_{n \mid q} n d(n) P_1(n,q) S\left( \overline{a}, 0 ; \frac{q}{n}\right) \\
    + \Tilde{v_x}\dblprime(1) \sum_{n \mid q} n d(n) P_0(n,q) S\left( \overline{a}, 0 ; \frac{q}{n}\right),
\end{multline}
and $V_\pm$ is the integral transform defined in \autoref{Gpm}. The constant $\mathsf{c}$ vanishes for the first case and it takes value $1$ for the latter. 
Now we extract the oscillations of $V_\pm$, for $yX^2 \gg X^{\varepsilon}$, we have
\begin{equation*}
    V_\pm(y) = \pi^4 y \int_{0}^\infty v_x(z) \sum_{j=1}^{K_0} \frac{c_j e{(3 (y z)^{1/3})} + d_je(-3(yz)^{1/3})}{(\pi^3 yz)^{j/3}} dz + \mathcal{O}(X^{-2025}),
\end{equation*}
where $K_0 = \floor{\frac{6075}{\varepsilon} +2} +1$ and $\floor{\cdot}$ denotes the greatest integer function. Substituting the expression of $v_x(z)$ and considering the term corresponding to $j=1$, we get
\begin{multline*}
    V_\pm(y) =  \frac{-2 \pi^3 y^{2/3}}{\sqrt{3\pi}} \int_{0}^\infty V\left(\frac{z}{X^2}\right) e\lr{\frac{zx}{q\mathcal{Q}}} z^{-1/3} \left(e(3 (y z)^{1/3}) - e(-3(yz)^{1/3})\right) dz \\+ \mathcal{O}(X^{-2025}),
\end{multline*}
as the other terms can be treated similarly and, in fact, give us better estimates. Using the change of variable $z \rightarrow X^2 z$ and putting $y = n^2 m / q^3$, up to some lower order terms, we arrive at
\begin{eqnarray} \label{osci Gpm}
    V_\pm\lr{\frac{n^2m}{q^3}} &=& \frac{-2\pi^3 }{ \sqrt{3 \pi}} \frac{(X^2 n^2 m)^{2/3}}{q^2} \int_0^\infty V(z) z^{-1/3}e\left( \frac{X^2 xz}{q\mathcal{Q}} + \frac{3(X^2 z n^2m)^{1/3}}{q}\right) dz \notag\\ 
    \hspace{1cm}&+& \frac{2\pi^3 }{ \sqrt{3 \pi}} \frac{(X^2 n^2 m)^{2/3}}{q^2} \int_0^\infty V(z) z^{-1/3}  e\left( \frac{X^2xz}{q\mathcal{Q}} - \frac{3(X^2 z n^2m)^{1/3}}{q}\right) dz \notag\\
   &=& \frac{X^{4/3}(n^2m)^{2/3}}{q^2}   \mathcal{I}_\pm(n^2 m, x,q),
\end{eqnarray}
where 
\begin{equation} \label{I integral}
    \mathcal{I}_\pm(n^2 m, x,q) =  \mathsf{c}_0 \int_0^\infty V_1^\pm(z) z^{-1/3} e\left( \frac{X^2 xz}{q \mathcal{Q}} \pm \frac{3(X^2 z n^2 m)^{1/3}}{q}\right) dz,
\end{equation}
with $V_1^\pm(z) = \pm V(z)$, and $\mathsf{c}_0 =  -\frac{2\pi^3 }{ \sqrt{3 \pi}}$. Using integration by parts repeatedly, the integral $\mathcal{I}_\pm(n^2 m, x,q)$ is negligibly small if 
\begin{equation} \label{M_0}
    n^2 m \gg \frac{q^3}{X^2}=: \mathsf{M}_0.
\end{equation}
For $n^2 m X^2 \ll q^3 X^\varepsilon$, the analysis of $V^\pm$ in \autoref{voronoi1} is straightforward.  
Using the expression in \autoref{Gpm}, it yields a sharper estimate for $\mathcal{S}^*(X,Y)$ (see \cite[Section~$5.9$]{HCIJNT}). Therefore, we conclude this subsection by summarizing the above discussion in the following lemma.

\medskip
\begin{lemma} \label{lemma vornoi}
    Let
    \begin{equation*}
        \mathcal{S}_1 = \sum_{r \in \mathbb{Z}} \mathcal{A}(r) e\lr{\frac{ar}{q}}v_x(r).
    \end{equation*}
    Then, for $n^2mX^2 \gg q^3 X^\varepsilon$, we have
    \begin{multline*}
        \mathcal{S}_1 = \frac{\mathsf{c}}{2q^2}\Tilde{\mathcal{V}}(v_x,a,q) \\ + \frac{X^{4/3}}{q} \sum_{\pm}\sum_{n \mid q} \sum_{n^2m \ll \mathsf{M}_0} \frac{B(n,m)}{n^{-1/3}m^{1/3}} S\lr{\ol{a}, \pm m; \frac{q}{n}} \mathcal{I}_\pm(n^2m,x,q) 
        + \mathcal{O}\lr{X^{-2025}},
    \end{multline*}
 where $v_x(y) = V\lr{\frac{y}{X^2}} e\lr{\frac{xy}{q\mathcal{Q}}}$, $\Tilde{\mathcal{V}}(v_x,a,q)$ is defined in \autoref{V tilde}, the constant $\mathsf{c}$ is $1$ in case of $d_3$ and $0$ otherwise, $\mathcal{I}_\pm$ is an integral transform defined as in \autoref{I integral}, $\mathsf{M}_0$ is defined in \autoref{M_0}, and coefficients $B(n,m)$ are defined in \autoref{B(n,m)}.    
\end{lemma}

\medskip

\noindent After applying the Voronoi summation formula to extract key oscillatory components, we now use the Poisson summation formula to handle the sums over $n_1$ and $n_2$.

\subsection{Poisson summation formula} \label{poisson section} Denote the sums over $n_1$ and $n_2$ in \autoref{SXY after CM} as $\mathcal{P}(X)$ and then
writing $n_i = \alpha_i + \ell_i q$ for $i=1,2$, we have
\begin{multline*}
    \mathcal{P}(X)= \sum_{\alpha_1 \textrm{ mod }q} \sum_{\alpha_2 \textrm{ mod }q} \mathop{\sum\sum}_{\ell_1,\ell_2 \in \mathbb{Z}} W_1\lr{\frac{\alpha_1+ \ell_1 q}{X}} W_2\lr{\frac{\alpha_2+\ell_2q}{Y}} \\ 
    \times e\lr{\frac{-aQ(\alpha_1,\alpha_2)}{q}} e\lr{\frac{-xQ(\alpha_1+ \ell_1q,\alpha_2+\ell_2q)}{q\mathcal{Q}}}.
\end{multline*}
Applying the Poisson summation formula to the sums over $\ell_1$ and $\ell_2$, we get 
\begin{eqnarray*}
 = \sum_{\alpha_1 \textrm{ mod } q} \sum_{\alpha_2 \textrm{ mod } q} e\left(\frac{-aQ(\alpha_1, \alpha_2)}{q}\right) \mathop{\sum\sum}_{m_1, m_2 \in \mathbb{Z}} \mathop{\iint}_{\mathbb{R}^2} W_1 \left(\frac{\alpha_1 + x q}{X}\right)  W_2\left(\frac{\alpha_2 + y q}{Y}\right) \\  \times e\left(\frac{-x Q(\alpha_1 + xq, \alpha_2 + yq)}{q \mathcal{Q}}\right)e(-m_1x - m_2y) \;dx \;dy.
\end{eqnarray*}
Substituting the change of variables $(\alpha_1 + x q)/X=u$ and $(\alpha_2 + y q)/Y=v$, we obtain
\begin{equation} \label{P(X)}
\mathcal{P}(X) = \frac{XY}{q^2} \mathop{\sum\sum}_{m_1, m_2 \in \mathbb{Z}} \mathfrak{C}(m_1,m_2,a,q) \;\;\mathfrak{J}(m_1,m_2,x,q) \
\end{equation}
where the character sum $\mathfrak{C}(m_1,m_2,a,q)$ is given by
\begin{equation*}
    \mathfrak{C}(m_1,m_2,a,q) = \sum_{\alpha_1 \textrm{ mod } q} \sum_{\alpha_2 \textrm{ mod }q} e\left(\frac{-aQ(\alpha_1, \alpha_2) + m_1\alpha_1 + m_2\alpha_2}{q}\right), 
\end{equation*}
and the integral transform $\mathfrak{J}(m_1,m_2,x,q)$ is given by 
\begin{align} \label{integral transform J}
    \mathfrak{J}(m_1,m_2,x,q)= \mathop{\iint}_{\mathbb{R}^2} W_1(u) W_2(v) e\left(\frac{-x Q(uX,vY)}{q\mathcal{Q}} \right) e\left(- \frac{m_1 uX}{q} - \frac{m_2vY}{q}\right)  du\;dv. 
\end{align}
By applying integration by parts $j$ times, and using $$\frac{\partial^j}{\partial u^j} e\left(\frac{-xQ(u X, vY)}{q \mathcal{Q}} \right) \ll \left(\frac{X}{q}\right)^j,$$ we obtain (for $m_1, m_2 \neq 0$) that sum over $m_1$ and $m_2$ is negligibly small unless $m_1 \ll RX^\varepsilon $ and $m_2 \ll R X^{1+\varepsilon}/Y$. Substituting the change of variables $\alpha_1 \rightarrow \overline{a}\alpha_1\;$ and $\alpha_2 \rightarrow \overline{a}\alpha_2$, where $\overline{a}$ is the multiplicative inverse of $a$ modulo $q$. The character sum $\mathfrak{C}(m_1,m_2,a,q)$ gets transformed to 
\begin{align} \label{char poisson}
    \mathfrak{C}(m_1,m_2,a;q) =  \sum_{\alpha_1 \textrm{ mod }q} \sum_{\alpha_2 \textrm{ mod } q} e\left(\frac{-\overline{a}\lr{Q(\alpha_1, \alpha_2) - m_1\alpha_1 - m_2 \alpha_2}}{q} \right).
\end{align}
Combining the expressions for sum $\mathcal{S}_1$ from \autoref{lemma vornoi}, and for $n_1,n_2$ sums from \autoref{P(X)}, we can write $\mathcal{S}^*(X,Y)$ as 
\begin{equation} \label{S split}
    \mathcal{S}^*(X,Y) = \mathcal{S}_M(X,Y) + \mathcal{S}_E(X,Y),
\end{equation}
where
\begin{align} \label{SXY main}
     \mathcal{S}_M(X,Y) &= \mathsf{c}\frac{XY}{2\mathcal{Q}} \sum_{1 \leqslant q \leqslant \mathcal{Q}} \frac{1}{q^5} \sum_{m_1 \ll RX^\varepsilon} \sum_{m_2 \ll RX^{1+\varepsilon}/Y} \ \sideset{}{^\star}\sum_{a \textrm{ mod } q} \mathfrak{C}(m_1,m_2,a;q) \notag\\
     &\hspace{2.5cm}\times \int_{\mathbb{R}} \psi(q,x) U(x) \Tilde{\mathcal{V}}(v_x,a,q) \mathfrak{J}(m_1,m_2,x,q) dx dy,
\end{align}
and
\begin{align} \label{SXY error}
    \mathcal{S}_E(X,Y) &= \frac{X^{7/3}Y}{\mathcal{Q}} \sum_{1 \leqslant q \leqslant \mathcal{Q}} \frac{1}{q^4} \; \sum_{m_1 \ll RX^\varepsilon} \sum_{m_2 \ll RX^{1+\varepsilon}/Y} \sum_{\pm} \sum_{n \mid q} \sum_{n^2 m \ll \mathsf{M}_0} \frac{B(n,m)}{n^{-1/3} m^{1/3}} \notag \\ 
    &\hspace{2.5cm}\times \sideset{}{^\star}\sum_{a \textrm{ mod } q} S\left(\overline{a}, \pm m; \frac{q}{n}\right) \mathfrak{C}(m_1,m_2,a;q) \notag \\ 
    &\hspace{2.5cm}\times \int_{\mathbb{R}} \psi(q,x)U(x) \mathcal{I}_{\pm}(n^2m,x,q) \mathfrak{J}(m_1,m_2,x,q)dx.
\end{align}
\noindent The main term will come from $\mathcal{S}_M$, when $m_1=m_2=0$ which we will analyze in the next subsection. First, we will obtain a bound for $\mathcal{S}_E(X,Y)$. 
\subsection{Bound for \texorpdfstring{$\mathcal{S}_E(X,Y)$}{}}
In this subsection, we will obtain cancellations in $\mathcal{S}_E$, which is defined in \autoref{SXY error}. For that we need to analyze the four-fold integrals present in $\mathcal{S}_E$. 
\subsubsection{ Simplification of integrals} \label{W integral section} 
Set
\begin{equation} \label{Wpm}
    \mathcal{W}^{\pm}(m_1,m_2,n,m,q) := \int_{\mathbb{R}} \psi(q,x) U(x) \mathcal{I}_{\pm}(n^2m,x,q) \mathfrak{J}(m_1,m_2,x,q) dx.
\end{equation}

\noindent We prove the following bound for it.

\medskip

\begin{lemma} \label{Wintegral}
We have
    \begin{equation*}
         \mathcal{W}^{\pm}(...) \ll \begin{cases}
         \vspace{0.2cm}
             \frac{R q^{3/2}X^\varepsilon}{\mathcal{Q}^{4/3} (n^2m)^{1/6}} \textrm{ with } m_1 \asymp \frac{(n^2m)^{1/3}}{X^{1/3}} & \textrm{ if } m_1 \neq 0 \textrm{ and } q \ll \mathcal{Q}^{1-2\varepsilon} \\ \vspace{0.2cm}
             \frac{q X^\varepsilon}{ \mathcal{Q}} & \textrm{ if } m_1 \neq 0 \textrm{ and } q \gg \mathcal{Q}^{1-2\varepsilon} \\
             \vspace{0.2cm}
            \frac{R q^{3/2}X^\varepsilon}{\mathcal{Q}^{4/3} (n^2m)^{1/6}} & \textrm{ if } m_1=0 
         \end{cases}.
    \end{equation*}
\end{lemma}
\begin{proof}
Substitute the expressions for $\mathcal{I}_{\pm}$ and $\mathfrak{J}$ from \autoref{I integral} and \autoref{integral transform J} into \autoref{Wpm}, we get the following four-fold integrals.
\begin{multline*}
     \mathcal{W}^{\pm}(...) = \mathsf{c}_0 \int_{\mathbb{R}} \; \psi(q,x) U(x) \int_{0}^{\infty} V_1^{\pm}(z) z^{-1/3} e\Biggl( \frac{X^2 x z}{q \mathcal{Q}} \pm \frac{3(X^2 z n^2 m)^{1/3}}{q}\Biggr) \\  
     \hfill \times \mathop{\iint}_{\mathbb{R}^2} W_1(u) W_2(v) e\Biggl(\frac{-xQ(uX, vY)}{q \mathcal{Q}}\Biggr) e\Biggl( -\frac{ m_1uX}{q} - \frac{m_2 vY}{q} \Biggr)  du dv dz dx. 
\end{multline*}
First consider the $x$-integral i.e.,
\begin{equation*}
    \int_{\mathbb{R}} \psi(q,x) U(x) e\Biggl( \frac{\lr{X^2 z - Q(uX, vY)} x}{q\mathcal{Q}} \Biggr) dx,
\end{equation*}
for small $q$, i.e., $q \ll \mathcal{Q}^{1-\varepsilon}$, we split the $x-$integral into two parts $| x | \ll \mathcal{Q}^{-\varepsilon}$ and $|x| \gg \mathcal{Q}^{-\varepsilon}$.  In the former case, we use \autoref{dfi} by which we can replace $\psi(q,x)$ by $1$ up to some negligible error term. We obtain
\begin{equation*}
    \int_{|x| \ll \mathcal{Q}^{-\varepsilon}} U(x) e\left( \frac{(X^2 z - Q(uX, vY))x }{q \mathcal{Q}}\right) dx,
\end{equation*}
then applying integration by parts repeatedly, we get that the above integral is negligible unless $q \mathcal{Q} X^\varepsilon \gg | X^2 z - Q(uX, v Y)|$ or $|z - Q(uX, vY)/X^2| \ll \frac{q\mathcal{Q}}{X^{2-\varepsilon}}$. Otherwise, we have 
\begin{equation*}
     \int_{|x|\gg \mathcal{Q}^{-\varepsilon} }\psi(q,x) U(x) e\Biggl( \frac{\lr{X^2 z - Q(uX, vY)} x}{q\mathcal{Q}} \Biggr) dx.
\end{equation*}
Again applying integration by parts repeatedly, and using the properties of $\psi$ and bump function $U$ given in \autoref{dfi} \autoref{delta3}, together with the following properties
\begin{equation*}
   \frac{\partial^j}{\partial x^j}\psi(q,x) \ll \mathcal{Q}^{\varepsilon j}, \quad \quad U^{j}(x) \ll \mathcal{Q}^{\varepsilon j}, \quad \quad \text{ for any } j\geqslant 1,
\end{equation*}
we get that the integral is negligibly small unless $|z - Q(uX, vY)/X^2| \ll \frac{q\mathcal{Q}}{X^{2-\varepsilon}}$. For large $q$, i.e., if $q \gg \mathcal{Q}^{1-\varepsilon}$, this condition holds trivially.\\

\noindent Let $z - Q(uX, vY)/X^2 =: s$ with $|s| \ll \frac{q\mathcal{Q}}{X^{2-\varepsilon}}$, we arrive at the following expression
\begin{multline}
    \mathcal{W}^{\pm}(...) = \mathsf{c}_0 \int_{\mathbb{R}} \; \psi(q,x) U(x) \int_{|s| \ll \frac{q\mathcal{Q}}{X^{2-\varepsilon}}} \mathop{\iint}_{\mathbb{R}^2} V_2^{\pm}\left(\frac{X^2s + Q(uX, vY)}{X^2}\right)W_1(u)  W_2(v)  \\ \times e\Biggl( \frac{sxX^2}{q\mathcal{Q}} \pm \frac{3\lr{ \lr{X^2s + Q(uX, vY)} n^2 m}^{1/3}}{q}\Biggr) e\Biggl( -\frac{ m_1 uX}{q} - \frac{m_2 vY}{q} \Biggr) dudv dsdx \\ \hfill + \text{ negligible error term}, 
\end{multline}
where $V_{2}^{\pm}(z)=V_1^\pm(z) z^{-1/3}$.
Now, we consider the $u$-integral 
\begin{equation*}
   \mathcal{U}:= \int_{\mathbb{R}} W_3(u)  e\Biggl( \frac{\pm3\lr{ \lr{X^2s + Q(uX, vY)} n^2 m}^{1/3}- m_1 uX}{q} \Biggr) du,
\end{equation*}
where $W_3(u) = V_2^{\pm}\left(\frac{X^2s + Q(uX, vY)}{X^2}\right)W_1(u)$. If $q \ll \mathcal{Q}^{1-2\varepsilon}$, then $|s| \ll X^{-\varepsilon}$. We have the following phase function in $u$:
\begin{equation*}
     \phi(u)=  2\pi i\Bigl( \pm \frac{3( Q(uX, vY) n^2 m)^{1/3}}{q} - \frac{m_1 uX}{q} \Bigr).
\end{equation*}
Using the stationary phase analysis given in \autoref{stationary phase} $(b)$ with $\phi\dblprime(u) \asymp \frac{X^{2/3}(n^2m)^{1/3}}{q}$, provided $m_1 \asymp \frac{(n^2m)^{1/3}}{X^{1/3}}$, we get
\begin{equation}
     \mathcal{U} \ll \frac{R\sqrt{q}}{X^{1/3}(n^2 m)^{1/6}},
\end{equation}
with $m_1 \asymp \frac{(n^2m)^{1/3}}{X^{1/3}}$. This restriction over $m_1$ ensures existence of the stationary point. 
Executing the remaining integrals trivially, we get 
\begin{equation*}
    \mathcal{W}^{\pm}(...) \ll \frac{q\mathcal{Q}}{X^{2-\varepsilon}} \cdot \frac{R \sqrt{q}}{X^{1/3} (n^2 m)^{1/6}}  = \frac{R q^{3/2}X^\varepsilon}{\mathcal{Q}^{4/3} (n^2 m)^{1/6}},
\end{equation*}
as $\mathcal{Q}=X$. Otherwise, if $q\gg \mathcal{Q}^{1-2\varepsilon}$, we treat the remaining integrals trivially. Now in last case, the phase function in $u$ integral is $\phi_1(u)= \frac{6\pi i( Q(uX, vY) n^2 m)^{1/3}}{q}$, using the second derivative bound as given in \autoref{exponential sum}, we have
\begin{equation} \label{U W bound}
    \mathcal{U} \ll \frac{R q^{1/2}}{X^{1/3}(n^2m)^{1/6}} \ \textrm{ and } \ \mathcal{W}^\pm(...) \ll \frac{R q^{3/2}X^\varepsilon}{\mathcal{Q}^{4/3} (n^2m)^{1/6}}.
\end{equation}
When both $m_1=m_2=0$, using the second derivative bound on the $u$ integral, treating $v$ integral trivially, we get the bound as in \autoref{U W bound}. This proves the lemma.
\end{proof}

\noindent Next, by opening up the Kloosterman sum, we get
\begin{align} \label{SXY after SF}
  \mathcal{S}_E(X,Y) &= \frac{X^{7/3}Y}{\mathcal{Q}} \sum_{1 \leqslant q \leqslant \mathcal{Q}} \frac{1}{q^4} \; \Biggl[ \sum_{\pm} \sum_{n \mid q} \ \sideset{}{^\star}\sum_{\beta \textrm{ mod } \frac{q}{n}} \sum_{n^2 m \ll \mathsf{M}_0} \frac{B(n,m)}{n^{-1/3} m^{1/3}}  e\left(\frac{\pm m \overline{\beta}}{q/n}\right) \notag\\ 
  &\hspace{1cm}\times \sum_{m_1 \ll RX^\varepsilon} \sum_{m_2 \ll RX^{1+\varepsilon}/Y} \ \sideset{}{^\star}\sum_{a \textrm{ mod } q} e\left( \frac{\overline{a} \beta}{q/n} \right) \mathfrak{C}(m_1,m_2,a;q) \Biggr] \notag \\
  &\hspace{1cm}\times \int_{\mathbb{R}} \psi(q,x) U(x) \mathcal{I}_{\pm}(n^2m,x,q) \; \mathfrak{J}(m_1,m_2,x,q) \;dx. 
\end{align}


     
\subsubsection{ Analyzing the Character sum}
Consider the sum 
\begin{align} \label{char sum quad form}
    \mathfrak{S}_2(m_1,m_2,n ;q) &:= \sideset{}{^\star}\sum_{\beta \textrm{ mod } \frac{q}{n}} \ \sideset{}{^\star}\sum_{a \textrm{ mod } q}   e\left(\frac{\overline{a}\beta }{q/n}\right) \; \mathfrak{C}(m_1,m_2,a;q) \notag\\ 
     &= \sideset{}{^\star}\sum_{\beta \textrm{ mod } \frac{q}{n}} \ \sideset{}{^\star}\sum_{a \textrm{ mod } q}  e\left(\frac{\overline{a} \beta}{q/n} \right) \notag\\ 
     & \times \sum_{\alpha_1 \textrm{ mod } q} \sum_{\alpha_2 \textrm{ mod } q} e\left(\frac{-\overline{a}(Q(\alpha_1,\alpha_2)-m_1\alpha_1-m_2\alpha_2)}{q} \right).
\end{align} 

\noindent We prove the following bound for the above character sum.

\medskip

\begin{lemma} \label{bound S1}
    We have
    \begin{equation}
        \mathfrak{S}_2(m_1,m_2,n;q) \ll \frac{q_1^3 q_2^2 }{n}  d(q_1) d(q_2),
    \end{equation}
where $q=q_1q_2$ with $q_1\mid (2n |\mathbf{A}|)^\infty$ and $(q_2, 2n|\mathbf{A}|q_1) =1$.  
\end{lemma}
\begin{proof}
Write $q=q_1q_2$ with $q_1\mid (2n |\mathbf{A}|)^\infty$ and $(q_2, 2n|\mathbf{A}|q_1) =1$. Substitute $\beta = \beta_1q_2 \overline{q_2} + \beta_2 \frac{q_1}{n}\overline{\frac{q_1}{n}}$, where $q_2 \overline{q_2} \equiv 1 \ \textrm{mod}\; q_1/n$ and $\frac{q_1}{n} \overline{\lr{\frac{q_1}{n}}} \equiv 1 \; \textrm{mod} \; q_2$. Similarly, split the sums over $a, \alpha_1$ and $ \alpha_2$ modulo $q_1$ and $q_2$. We can now write $\mathfrak{S}_2(...)$ as
\begin{align*}
     & \sideset{}{^\star}\sum_{\beta_1 \textrm{ mod } \frac{q_1}{n}} \ \sideset{}{^\star}\sum_{a_1 \textrm{ mod } q_1}  e\left( \frac{\overline{a_1}\overline{q_2}\beta_1}{q_1/n} \right) \sum_{\alpha_1^\prime \textrm{ mod } q_1} \sum_{\alpha_2^\prime \textrm{ mod } q_1} \\
     &\hspace{1cm}\times e\left(\frac{-\overline{a_1}\overline{q_2} \lr{Q(\alpha^\prime_1,\alpha^\prime_2)-m_1\alpha^\prime_1-m_2\alpha^\prime_2}}{q_1} \right)
   \sideset{}{^\star}\sum_{\beta_2 \textrm{ mod } q_2} \ \sideset{}{^\star}\sum_{a_2 \textrm{ mod } q_2}  e\left( \frac{\overline{a_2} (\overline{q_1/n})\beta_2 }{q_2}\right) \\
   &\hspace{1cm}\times \sum_{\alpha_1 \textrm{ mod } q_2} \sum_{\alpha_2 \textrm{ mod } q_2} e\left( \frac{-\overline{a_2} \overline{q_1}\lr{Q(\alpha_1,\alpha_2)- m_1\alpha_1- m_2\alpha_2}}{q_2}\right).
\end{align*}
First, consider the last line of above equation,
\begin{align*}
    &\sideset{}{^\star}\sum_{\beta_2 \textrm{ mod } q_2} \ \sideset{}{^\star}\sum_{a_2 \textrm{ mod } q_2}  e\left( \frac{\overline{a_2}(\overline{q_1/n})\beta_2 }{q_2}\right) \hspace{-0.2cm}\sum_{\alpha_1 \textrm{ mod } q_2} \sum_{\alpha_2 \textrm{ mod } q_2} \\ 
    &\hspace{5.5cm}\times e\left( \frac{-\overline{a_2q_1}\lr{Q(\alpha_1,\alpha_2)- m_1\alpha_1- m_2\alpha_2}}{q_2}\right) \\
    &\hspace{1cm}= \sideset{}{^\star}\sum_{\beta_2 \textrm{ mod } q_2}\; \sideset{}{^\star}\sum_{a_2 \textrm{ mod } q_2} q_2 \; \varepsilon^2_{q_2} \left(\frac{|\mathbf{A}|}{q_2}\right) e\left(\frac{\overline{a_2}\left(\ol{(q_1/n)}\beta_2  + \ol{q_1}NQ^*(m_1,m_2)\right)}{q_2}\right).
\end{align*}
Here we have used \autoref{gauss sum} to the sums over $\alpha_1$ and $\alpha_2$. Now, the sum over $a_2$ is the Ramanujan sum, we can write the above equation as
\begin{align*}
    &= q_2 \; \varepsilon^2_{q_2} \left(\frac{|\mathbf{A}|}{q_2}\right) \sum_{d_2 \mid q_2} d_2 \mu\left(\frac{q_2}{d_2}\right) \sideset{}{^\star}\sum_{\substack{\beta_2 \textrm{ mod } q_2 \\ \beta_2 \equiv -N Q^*(m_1,m_2) \textrm{ mod } d_2}} 1 \\ 
    &\ll q_2 \sum_{d_2 \mid q_2} d_2 \starsum_{\substack{\beta_2 \textrm{ mod } q_2 \\ \beta_2  \equiv -N Q^*(m_1,m_2) \textrm{ mod } d_2}} 1 \\
    &\ll \; q_2  \sum_{d_2 \mid q_2} d_2 \frac{q_2}{d_2} \ll q_2^2 d(q_2).
\end{align*}
Now consider 
\begin{eqnarray*}
    &&\sideset{}{^\star}\sum_{\beta_1\textrm{ mod }\frac{q_1}{n}} \ \sideset{}{^\star}\sum_{a_1 \textrm{ mod } q_1}   e\left( \frac{\overline{a_1}\overline{q_2}\beta_1}{q_1/n} \right) \sum_{\alpha_1^\prime \textrm{ mod } q_1} \sum_{\alpha_2^\prime \textrm{ mod } q_1} \\
    &&\hspace{6.3cm}\times e\left(\frac{-\overline{a_1}\overline{q_2}\lr{Q(\alpha^\prime_1,\alpha^\prime_2)-m_1\alpha^\prime_1-m_2\alpha^\prime_2}}{q_1} \right) \\
   &&\hspace{0.1cm}= \sideset{}{^\star}\sum_{\beta_1 \textrm{ mod } \frac{q_1}{n}}   \sum_{\alpha_1^\prime \textrm{ mod } q_1} \sum_{\alpha_2^\prime \textrm{ mod }q_1} \ \sideset{}{^\star}\sum_{a_1 \textrm{ mod } q_1} \hspace{-0.2cm} e\left(\frac{\overline{a_1}\overline{q_2}(-Q(\alpha^\prime_1,\alpha^\prime_2)+m_1\alpha^\prime_1+m_2\alpha^\prime_2 + \beta_1 n)}{q_1} \right) \\
    &&\hspace{0.2cm}\ll \sum_{\alpha_1^\prime \textrm{ mod } q_1} \sum_{\alpha_2^\prime \textrm{ mod } q_1} \; \sideset{}{^\star}\sum_{\beta_1 \textrm{ mod } \frac{q_1}{n}} \; \sum_{d_1 \mid (q_1,-Q(\alpha^\prime_1,\alpha^\prime_2)+m_1\alpha^\prime_1+m_2\alpha^\prime_2 + \beta_1 n) } d_1 \mu\left(\frac{q_1}{d_1}\right)\\
   &&\hspace{0.2cm}\ll \sum_{\alpha_1^\prime \textrm{ mod } q_1} \sum_{\alpha_2^\prime \textrm{ mod } q_1} \sum_{d_1 \mid q_1} d_1 \sideset{}{^\star}\sum_{\substack{ \beta_1 \textrm{ mod } \frac{q_1}{n} \\ m_1\alpha^\prime_1+m_2\alpha^\prime_2 + \beta_1 n \equiv Q(\alpha^\prime_1,\alpha^\prime_2) \textrm{ mod }d_1}} 1 \\
   &&\hspace{0.2cm}\ll \sum_{\alpha_1^\prime \textrm{ mod } q_1} \sum_{\alpha_2^\prime \textrm{ mod } q_1} \sum_{d_1 \mid q_1} d_1  \frac{q_1}{n d_1} \ll \frac{q^3_1}{n}d(q_1).
\end{eqnarray*}
\end{proof}

\noindent Recall from the \autoref{SXY after SF}, the sum $\mathcal{S}_E(X,Y)$ is given as
\begin{eqnarray*} 
 \mathcal{S}_E(X,Y) &&= \frac{X^{7/3}Y}{\mathcal{Q}} \sum_{m_1 \ll R X^\varepsilon} \sum_{m_2 \ll RX^{1+\varepsilon}/Y } \sum_{1 \leqslant q \leqslant \mathcal{Q}} \frac{1}{q^4}\\
    &&\hspace{1cm}\times \sum_{n \mid q} \ \sideset{}{^\star}\sum_{\beta \textrm{ mod } \frac{q}{n}} \Biggl[ \sum_{\pm} 
     \sum_{n^2 m \ll \mathsf{M}_0} \frac{B(n,m)}{n^{-1/3} m^{1/3}} e\left(\frac{\pm m \overline{\beta}}{q/n}\right)\Biggr] \\
     &&\hspace{1cm}\times \sideset{}{^\star}\sum_{a \textrm{ mod } q} e\left(\frac{\overline{a}\beta }{q/n}\right) \mathfrak{C}(m_1,m_2,a;q) \mathcal{W}^{\pm}(m_1,m_2,n,m,q).
\end{eqnarray*}
Using bounds for $\mathcal{W}^{\pm}$ and $\mathfrak{S}_2$ from \autoref{Wintegral} and \autoref{bound S1}, respectively, we can have the following estimate for $\mathcal{S}_E$.
\begin{eqnarray*}
    &&\ll \frac{X^{7/3+\varepsilon}Y}{\mathcal{Q}} \Biggl( \ \sum_{m_2 \ll       RX^{1+\varepsilon}/Y } \sum_{\substack{1 \leqslant q \ll                    \mathcal{Q}^{1-2\varepsilon} \\ q=q_1q_2}} \frac{1}{q^4}  \sum_{n \mid q_1 \mid (2n|\mathbf{A}|)^\infty} n^{1/3} \frac{q^2 q_1}{n} \sum_{ m \ll            \mathsf{M}_0/n^2} \frac{|B(n,m)|}{ m^{1/3}}\\
        &&\hspace{0.5cm}\times \sum_{m_1 \asymp \frac{(n^2m)^{1/3}}{X^{1/3}}} \frac{R q^{3/2}}{\mathcal{Q}^{4/3} (n^2m)^{1/6}} + \sum_{m_1 \ll R X^\varepsilon} \sum_{m_2 \ll R X^{1+\varepsilon}/Y} \sum_{ \substack{\mathcal{Q}^{1-2\varepsilon} \ll q \leqslant \mathcal{Q} \\ q=q_1q_2}} \frac{1}{q^4} \\ 
        &&\hspace{0.5cm}\times \sum_{n \mid q_1 \mid (2n)^\infty} n^{1/3} \frac{q^2 q_1}{n} \sum_{n^2m \ll \mathsf{M}_0} \frac{|B(n,m)|}{m^{1/3}}
        + \sum_{m_2 \ll RX^{1+\varepsilon}/Y} \sum_{\substack{1 \leqslant q \leqslant \mathcal{Q} \\ q=q_1q_2}} \frac{1}{q^4} \\ 
        &&\hspace{0.5cm}\times \sum_{n \mid q_1 \mid (2n)^\infty} n^{1/3} \frac{q^2q_1}{n} \sum_{m \ll \mathsf{M}_0/n^2} \frac{|B(n,m)|}{m^{1/3}} \frac{Rq^{3/2}}{\mathcal{Q}^{4/3} (n^2m)^{1/6}}\Biggr)\\
    &&\ll \frac{X^{7/3+\varepsilon}Y}{\mathcal{Q}} \Biggl( \frac{R^2 X}{Y} \frac{\mathsf{M}_0^{1/2}}{\mathcal{Q}^{4/3} X^{1/3}} \sum_{q_1 \ll \mathcal{Q}^{1-2\varepsilon}} \sum_{n \mid q_1 \mid (2n)^\infty} \frac{q_1^{1/2}}{n^{4/3}} \sum_{q_2 \ll \frac{\mathcal{Q}^{1-2\varepsilon}}{q_1}} \frac{1}{q_2^{1/2}} \\
        &&\hspace{0.5cm}\times \lr{\sum_{m \ll \mathsf{M}_0/n^2} \frac{|B(n,m)|^2}{m^{1/3}}}^{1/2} + \frac{R^2X \mathsf{M}_0^{1/2}}{Y}  \sum_{q_1 \gg \mathcal{Q}^{1-2\varepsilon}} \sum_{n \mid q_1 \mid (2n)^\infty} \frac{1}{q_1 n^{5/3}} \\
        &&\hspace{0.5cm}\times \sum_{q_2 \gg \mathcal{Q}^{1-2\varepsilon}/q_1} \frac{1}{q_2^2}  \lr{\sum_{m \ll \mathsf{M}_0/n^2} \frac{|B(n,m)|^2}{m^{2/3}}}^{1/2} + \frac{R^2 X \mathsf{M}_0^{1/2}}{Y\mathcal{Q}^{4/3}} \sum_{q_1 \leqslant \mathcal{Q}} \sum_{n \mid q_1 \mid (2n)^\infty} \frac{q_1^{1/2}}{n^2} \\
        &&\hspace{0.5cm}\times\sum_{q_2 \leqslant \mathcal{Q}/q_1} \frac{1}{q_2^{1/2}}  \lr{\sum_{m \ll \mathsf{M}_0/n^2} \frac{|B(n,m)|^2}{m}}^{1/2} \Biggr) \\
    &&\ll \frac{X^{7/3+\varepsilon} Y}{\mathcal{Q}} \Biggl( \frac{R^2 X^{2/3} \mathsf{M}_0^{1/2+ 1/3}}{Y \mathcal{Q}^{4/3}} \sum_{q_1 \ll \mathcal{Q}^{1-2\varepsilon}} \sum_{n \mid q_1 \mid (2n)^\infty} \frac{q_1^{1/2}}{n^{2-\theta}} \frac{\mathcal{Q}^{1/2-\varepsilon}}{q_1^{1/2}} + \frac{R^2 X \mathsf{M}_0^{1/2+1/6}}{Y\mathcal{Q}^{1-2\varepsilon}} \\ 
    &&\hspace{0.5cm}\times \sum_{q_1 \gg \mathcal{Q}^{1-2\varepsilon}} \sum_{n \mid q_1 \mid (2n)^\infty} \frac{1}{n^{2-\theta}} + \frac{R^2 X \mathsf{M}_0^{1/2}}{Y \mathcal{Q}^{4/3}} \sum_{q_1 \leqslant \mathcal{Q}} \sum_{n \mid q_1 \mid (2n)^\infty} \frac{1}{n^{2-\theta}} \mathcal{Q}^{1/2}\Biggr) \ll X^{2+\varepsilon} R^2.
\end{eqnarray*}
Here, we have used \autoref{RB GL3}, \autoref{RB d3} and the fact that the sum over $q_1$ runs over divisors of $(2n|\mathbf{A}|)^\infty$. Overall, we have saved the size $XYR^{-2}$. A little extra saving will give us non-trivial cancellations in the sum $\mathcal{S}_E$.

\subsection{Final estimates}
From subsection \autoref{W integral section}, we have
\begin{equation}
\frac{\partial}{\partial s} \mathcal{W}^\pm(m_1,m_2,n,s,q) \ll 
\begin{cases}
\vspace{0.2cm}
    \frac{R q^{1/2} n^{1/3}}{\mathcal{Q}^{2/3}} \frac{X^\varepsilon}{s^{5/6}} & \ \textrm{ if } m_1 \neq 0 \textrm{ and } q \ll \mathcal{Q}^{1-2\varepsilon} \textrm{ or if } m_1=0 \\
     \vspace{0.2cm}
     \frac{n^{2/3}X^\varepsilon}{X^{1/3} s^{2/3}} & \ \textrm{ if } m_1 \neq 0 \textrm{ and } q \gg \mathcal{Q}^{1-2\varepsilon} \\
    \end{cases}.
\end{equation}
Using the above equation, we can bound the following. In the first case we have 
\begin{align} 
    &\sum_{n \mid q} \sum_{n^2 m \ll \mathsf{M}_0} \frac{B(n,m)}{n^{-1/3} m^{1/3}} e\left(\frac{ m \overline{\beta}}{q/n}\right)\mathcal{W}^{\pm}(...)  \notag \\ 
    &\hspace{1cm}=\sum_{n \mid q} n^{1/3} \biggl[ - \int_1^{\mathsf{M}_0/n^2} \Biggl\{ \sum_{m\leq t} B(n,m) e\left(\frac{ m\overline{\beta}}{q/n}\right) \Biggr\} \frac{\partial}{\partial t}\left( \frac{1}{t^{1/3}} \mathcal{W}^\pm(.,n,t,.)\right) dt \notag\\  
   &\hspace{2cm}+ \sum_{m \ll \mathsf{M}_0/n^2} B(n,m) e\left(\frac{m\overline{\beta}}{q/n} \right) \left(\frac{n^2}{\mathsf{M}_0}\right)^{1/3} \mathcal{W}^\pm(.,n,\mathsf{M}_0/n^2,.) \biggr] \label{B(n,m) partial sum}\\
    &\ll \sum_{n \mid q} n^{1/3} \biggl[ \int_1^{\mathsf{M}_0/n^2} \frac{t^{3/4}}{t^{4/3}} \frac{R X^\varepsilon q^{3/2}}{\mathcal{Q}^{4/3} (n^2 t)^{1/6}} dt \notag\\ 
    &\hspace{1cm}+ \int_1^{\mathsf{M}_0/n^2} \frac{t^{3/4}}{t^{1/3}} \frac{R X^\varepsilon q^{1/2} n^{1/3}}{\mathcal{Q}^{2/3}} \frac{1}{t^{5/6}} \; dt + \left(\frac{\mathsf{M}_0}{n^2}\right)^{3/4} \frac{n^{2/3}}{\mathsf{M}_0^{ 1/3}} \frac{R X^\varepsilon q^{3/2}}{\mathcal{Q}^{4/3} \mathsf{M}_0^{1/6}}\biggr] \notag\\
    &\ll R X^\varepsilon \sum_{n \mid q} n^{1/3} \biggl[ \frac{q^{3/2}}{\mathcal{Q}^{4/3} n^{1/3} } \left(\frac{\mathsf{M}_0 }{n^2}\right)^{1/4} + \frac{q^{1/2} n^{1/3} }{\mathcal{Q}^{2/3}} \left(\frac{\mathsf{M}_0}{n^2}\right)^{7/12} + \frac{q^{3/2} \mathsf{M}_0^{1/4}}{n^{5/6} \mathcal{Q}^{4/3}}\biggr] \notag\\
    &\ll RX^\varepsilon \sum_{n \mid q} n^{1/3} \biggl[ \frac{q^{3/2} }{X^{13/12} n^{5/6}} + \frac{q^{1/2}}{X^{1/12}n^{5/6}} + \frac{q^{3/2}}{X^{13/12} n^{5/6}}\biggr] \notag\\
    &\ll R X^\varepsilon \sum_{n \mid q} \frac{1}{n^{1/2}} \frac{q^{3/2}}{X^{13/12}}. \label{bound after partial sum case1 }
\end{align}
If $m_1\neq 0$ and $q \gg \mathcal{Q}^{1-2\varepsilon}$, we can bound \autoref{B(n,m) partial sum} as
\begin{align} \label{bound after partial sum case2 }
    &\ll \sum_{n \mid q} n^{1/3} \biggl[ \int_1^{\mathsf{M}_0/n^2} \frac{t^{3/4}}{t^{4/3}} \frac{ q X^\varepsilon }{\mathcal{Q} } \;dt + \int_1^{\mathsf{M}_0/n^2} \frac{t^{3/4}}{t^{1/3}} \frac{n^{2/3}X^\varepsilon}{X^{1/3} t^{2/3}} \; dt \\
    &\hspace{5.8cm}+ \left(\frac{\mathsf{M}_0}{n^2}\right)^{3/4} \frac{n^{2/3}}{\mathsf{M}_0^{ 1/3}} \frac{q X^\varepsilon}{\mathcal{Q}} \biggr]\notag \\
    &\ll X^\varepsilon \sum_{n \mid q} n^{1/3} \Biggl[ \frac{q}{\mathcal{Q}} \lr{\frac{\mathsf{M}_0}{n^2}}^{5/12} + \frac{n^{2/3}}{X^{1/3}} \lr{\frac{\mathsf{M}_0}{n^2}}^{3/4} + \frac{q}{\mathcal{Q}} \frac{\mathsf{M}_0^{5/12}}{n^{5/6}}\Biggr] \notag\\
    &\ll \frac{\mathsf{M}_0^{5/12}X^\varepsilon}{\mathcal{Q}} \sum_{n \mid q} \frac{q}{n^{1/2}} + \frac{\mathsf{M}_0^{3/4}X^\varepsilon}{X^{1/3}} \sum_{n \mid q} \frac{1}{n^{1/2}} \ll \frac{\mathsf{M}_0^{5/12}X^\varepsilon}{\mathcal{Q}} \sum_{n \mid q} \frac{q}{n^{1/2}}. 
\end{align}
Now we estimate the sum $\mathcal{S}_E$ as follows.
\begin{align*}
  \mathcal{S}_{E}(X,Y) &\ll \frac{X^{7/3+\varepsilon}Y}{\mathcal{Q}} \mathop{\sum}_{ m_2 \ll RX^{1+\varepsilon}/Y } \sum_{ q \ll \mathcal{Q}^{1-2\varepsilon}} \frac{1}{q^4} \sum_{n \mid q} \ \sideset{}{^\star}\sum_{\beta \textrm{ mod } \frac{q}{n}} \ \sideset{}{^\star}\sum_{a \textrm{ mod } q} e\left(\frac{\overline{a}\beta }{q/n}\right)  \\
  &\hspace{1cm}\times  \mathfrak{C}(...)\sum_{\pm}  \; 
     \sum_{ m \ll \mathsf{M}_0/n^2} \sum_{m_1 \asymp \frac{(n^2m)^{1/3}}{X^{1/3}}} \frac{B(n,m)}{n^{-1/3} m^{1/3}} e\left(\frac{\pm m \overline{\beta}}{q/n}\right)\mathcal{W}^{\pm}(...)  \\ 
     &+ \frac{X^{7/3+\varepsilon}Y}{\mathcal{Q}} \sum_{m_1 \ll RX^\varepsilon} \mathop{\sum}_{ m_2 \ll RX^{1+\varepsilon}/Y } \sum_{ q \gg \mathcal{Q}^{1-2\varepsilon}} \frac{1}{q^4} \sum_{n \mid q} \ \sideset{}{^\star}\sum_{\beta \textrm{ mod } \frac{q}{n}} \ \sideset{}{^\star}\sum_{a \textrm{ mod } q} e\left(\frac{\overline{a}\beta }{q/n}\right) \\
     &\hspace{1cm}\times \mathfrak{C}(...) \sum_{\pm}  \sum_{ m \ll \mathsf{M}_0/n^2} \frac{B(n,m)}{n^{-1/3} m^{1/3}} e\left(\frac{\pm m \overline{\beta}}{q/n}\right)\mathcal{W}^{\pm}(...) \\
     &+ \frac{X^{7/3+\varepsilon}Y}{\mathcal{Q}} \mathop{\sum}_{ m_2 \ll RX^{1+\varepsilon}/Y } \sum_{1 \leqslant q \leqslant \mathcal{Q}} \frac{1}{q^4} \sum_{n \mid q} \ \sideset{}{^\star}\sum_{\beta \textrm{ mod } \frac{q}{n}} \ \sideset{}{^\star}\sum_{a \textrm{ mod } q} e\left(\frac{\overline{a}\beta }{q/n}\right) \mathfrak{C}(...) \\
     &\hspace{1cm}\times  \sum_{\pm}  \sum_{ m \ll \mathsf{M}_0/n^2} \frac{B(n,m)}{n^{-1/3} m^{1/3}} e\left(\frac{\pm m \overline{\beta}}{q/n}\right)\mathcal{W}^{\pm}(...).
\end{align*}
Using \autoref{bound S1} and equations \autoref{bound after partial sum case1 } and \autoref{bound after partial sum case2 }, we get
\begin{align*}
   \mathcal{S}_E(X,Y) &\ll \frac{X^{7/3+\varepsilon}Y}{\mathcal{Q}} \frac{R^2 X}{Y} \Biggl[ \ \sum_{ \substack{q \ll \mathcal{Q}^{1-2\varepsilon} \\ q=q_1q_2}} \frac{1}{q^4} \sum_{n \mid q} \frac{1}{n^{1/2}} \frac{q^{3/2}}{X^{13/12}} \frac{q^2 q_1 }{n} \\
   &\hspace{0.2cm}+ \frac{\mathsf{M}_0^{5/12}}{\mathcal{Q}} \sum_{ \substack{q \gg \mathcal{Q}^{1-2\varepsilon} \\ q=q_1q_2}} \frac{1}{q^3} \sum_{n \mid q} \frac{1}{n^{1/2}} \frac{q^2 q_1}{n} + \sum_{\substack{1 \leqslant q \leqslant \mathcal{Q} \\ q=q_1q_2}} \frac{1}{q^4} \sum_{n \mid q} \frac{1}{n^{1/2}} \frac{q^{3/2}}{X^{13/12}} \frac{q^2q_1}{n} \Biggr]\\
   &\ll \frac{X^{10/3+\varepsilon}R^2}{\mathcal{Q}} \Biggl[ \frac{1}{X^{13/12}} \sum_{\substack{q_1 \ll \mathcal{Q}^{1-2\varepsilon} \\ q_1 \mid (2 |\mathbf{A}|)^\infty}} q_1^{1/2} \sum_{n \mid q_1} \frac{1}{n^{3/2}} \sum_{q_2 \ll \mathcal{Q}^{1-2\varepsilon}/q_1 } \frac{1}{q_2^{1/2}} \\
   &+ \frac{\mathsf{M}_0^{5/12}}{\mathcal{Q}} \sum_{\substack{ \mathcal{Q}^{1-2\varepsilon} \ll q_1 \leqslant \mathcal{Q}_1 \\ q_1 \mid (2 |\mathbf{A}|)^\infty}} \sum_{n \mid q_1} \frac{1}{n^{3/2}} \sum_{\mathcal{Q}^{1-2\varepsilon}/q_1 \ll q_2 \leqslant \mathcal{Q}/q_1} \frac{1}{q_2} \Biggr]\\
   &\ll \frac{X^{9/4+\varepsilon}R^2}{\mathcal{Q}} \sum_{\substack{q_1 \ll \mathcal{Q}^{1-2\varepsilon} \\ q_1 \mid (2 |\mathbf{A}|)^\infty}} q_1^{1/2} \sum_{n \mid q_1} \frac{1}{n^{3/2}} \frac{\mathcal{Q}^{1/2}}{q_1^{1/2}} + \frac{X^{10/3+\varepsilon} \mathsf{M}_0^{5/12} R^2}{\mathcal{Q}^2}.
\end{align*}
In the last step, the sum over $q_1$ is bounded as it runs over divisors of $(2n|\mathbf{A}|)^\infty$. Also, recall that $\mathsf{M}_0= q^3/X^2 \ll X$ and $\mathcal{Q}=X$. Finally, we have the following bound for the sum $\mathcal{S}_E$
\begin{equation} \label{Bound S}
    \mathcal{S}_E(X,Y) \ll X^{7/4+\varepsilon} R^2. 
\end{equation}
This proves \autoref{GL3 FC}.
\subsection{Simplification of \texorpdfstring{$\mathcal{S}_M(X,Y)$}{}} \label{main term section}
The sum $\mathcal{S}_M(X,Y)$ will survive only in the case of \autoref{d3 thm}. In which case, we can write \autoref{SXY main} as the sum of the following two sums.
\begin{equation} \label{SM(X,Y)}
    \mathcal{S}_{M}(X,Y) = \mathcal{S}_{main}(X,Y) + \mathcal{S}_{err}(X,Y),
\end{equation}
where 
\begin{equation} \label{S main}
    \mathcal{S}_{main}(X,Y) = \frac{XY}{2\mathcal{Q}} \sum_{1 \leqslant q \leqslant \mathcal{Q}} \frac{1}{q^5} \int_{\mathbb{R}} \psi(q,x) U(x) \mathfrak{C}_M(0,0,x;q) \mathfrak{J}(0,0,x,q) dx,
\end{equation}
and 
\begin{multline} \label{S error}
    \mathcal{S}_{err}(X,Y) = \frac{XY}{2\mathcal{Q}} \sum_{1 \leqslant q \leqslant \mathcal{Q}} \frac{1}{q^5} \mathop{\sum_{m_1 \ll RX^\varepsilon} \sum_{m_2 \ll RX^{1+\varepsilon}/Y}}_{m_1^2+m_2^2 \neq 0} \ \\
     \times \int_{\mathbb{R}} \psi(q,x) U(x) \mathfrak{C}_M(m_1,m_2,x;q) \mathfrak{J}(m_1,m_2,x,q) dx.
\end{multline}
Here
\begin{align*}
    \mathfrak{C}_M(m_1,m_2,x,q) &= \sideset{}{^\star}\sum_{a \textrm{ mod } q} \mathfrak{C}(m_1,m_2,a;q) \Tilde{\mathcal{V}}(v_x,a,q) \notag\\
    &= \Tilde{v_x}(1) \sum_{n \mid q} n d(n) P_2(n,q) \mathfrak{C}_1(m_1,m_2,n;q) \notag\\ 
    &\hspace{1.5cm}+ \Tilde{v_x}^\prime(1) \sum_{n \mid q} n d(n) P_1(n,q) \mathfrak{C}_1(m_1,m_2,n;q) \notag\\
    &\hspace{3cm}+ \Tilde{v_x}^\dblprime(1) \sum_{n \mid q} n d(n) P_0(n,q) \mathfrak{C}_1(m_1,m_2,n;q),
\end{align*}
with \begin{equation} \label{C1 main CS}
     \mathfrak{C}_1(m_1,m_2,n;q) = \sideset{}{^\star}\sum_{a \textrm{ mod } q} \mathfrak{C}(m_1,m_2,a;q) S\lr{\ol{a},0;\frac{q}{n}}.
    \end{equation}
    
\noindent Also, recall from equations \autoref{P1} and \autoref{P2}, 
\begin{equation*}
    P_j(n,q) \ll (\log(n+2)(q+2))^j, \ \ j=0,1,2.
\end{equation*}
The following lemma gives us a bound for the character sum present in \autoref{C1 main CS}.

\medskip

\begin{lemma} \label{CS C2 main}
We have 
\begin{equation*}
    \mathfrak{C}_1(m_1,m_2,n;q) \ll \begin{cases}
        q (q,Q^*(m_1,m_2)) & \textrm{ if } m_1^2+m_2^2 \neq 0 \\
        \frac{q^2 q_1}{n} d(q_1)d(q_2) & \textrm{ if } m_1=m_2=0
    \end{cases}
\end{equation*}
where $Q^*$ is the adjoint quadratic form of $Q$ and $q=q_1q_2$ with $q_1 \mid (2n|\mathbf{A}|)^\infty$ and $(q_2,2n|\mathbf{A}|q_1)=1$.
\end{lemma}
\begin{proof}
    Substituting the expression for $\mathfrak{C}$ from \autoref{char poisson}, we get
    \begin{align*}
        \mathfrak{C}_1(...) &= \starsum_{a \textrm{ mod } q} \ \starsum_{\beta \textrm{ mod } \frac{q}{n}} e\lr{\frac{\ol{a}\beta}{q/n}} \\
        &\hspace{3cm}\times \sum_{\alpha_1 \textrm{ mod } q} \sum_{\alpha_2 \textrm{ mod } q}  e\lr{\frac{-\ol{a}(Q(\alpha_1,\alpha_2) -m_1\alpha_1 - m_2 \alpha_2)}{q}}. 
    \end{align*}
    Using \autoref{gauss sum}, assuming $q$ to be odd as for even $q$ we will get similar expression, we obtain
    \begin{align*}
        \mathfrak{C}_1(...) &= q \epsilon_q^2 \lr{\frac{|\mathbf{A}|}{q}} \starsum_{a \textrm{ mod } q}  \ \starsum_{\beta \textrm{ mod } \frac{q}{n}} e\lr{\frac{\ol{a}\beta}{q/n}} e\lr{\frac{\ol{a}Q^*(m_1,m_2)}{q}} \\
        &= q \varepsilon_q^2 \lr{\frac{|\mathbf{A}|}{q}} \mu\lr{\frac{q}{n}} 
            \starsum_{a \textrm{ mod } q} e\lr{\frac{\ol{a}Q^*(m_1,m_2)}{q}}.
    \end{align*}
When $m_1=m_2=0$, we proceed as in \autoref{bound S1}. Writing $q=q_1q_2$ with $q_1 \mid (2n|\mathbf{A}|)^\infty$ and $(q_2,2n|\mathbf{A}|q_1)=1$ then splitting the sum over $a, \beta, \alpha_1$ and $\alpha_2$ we get
    \begin{multline*}
        \mathfrak{C}_1(...)= \starsum_{a_1 \textrm{ mod } q_1} \ \starsum_{\beta_1 \textrm{ mod } \frac{q_1}{n}} \sum_{\alpha_1 \textrm{ mod } q_1} \sum_{\alpha_2 \textrm{ mod } q_1} e\lr{\frac{\ol{a_1 q_2} \beta_1}{q_1/n}} e\lr{\frac{-\ol{a_1 q_2} Q(\alpha_1,\alpha_2)}{q_1}} \\ \times \starsum_{a_2 \textrm{ mod } q_2} \ \starsum_{\beta_2 \textrm{ mod } q_2} \sum_{\alpha_1^\prime \textrm{ mod } q_2} \sum_{\alpha_2^\prime \textrm{ mod } q_2} e\lr{\frac{\ol{a_2 (q_1/n)}\beta_2}{q_2}} e\lr{\frac{-\ol{a_2 q_1}Q(\alpha_1^\prime,\alpha_2^\prime) }{q_2}}.
    \end{multline*}
Analogous to the proof of \autoref{bound S1}, we have the sums mod $q_1$ in the first line above are bounded by $q_1^3d(q_1)/n$ and the sums in the second line of the above equation are bounded by $q_2^2 d(q_2)$. Hence, we get the desired bound.    
\end{proof}

\noindent Now, to get a bound for $\mathcal{S}_{err}$, we have to analyze the integrals present in it. Recall that
\begin{multline*}
    \mathcal{S}_{err}(X,Y) = \frac{XY}{2\mathcal{Q}} \sum_{1\leqslant q \leqslant \mathcal{Q}} \frac{1}{q^5} \mathop{\sum_{m_1 \ll RX^\varepsilon} \sum_{m_2 \ll RX^{1+\varepsilon}/Y}}_{m_1^2+m_2^2 \neq 0} \sum_{n \mid q} n d(n) \mathfrak{C}_1(m_1,m_2,n;q) \\
    \times \sum_{j=0}^2 P_{2-j}(n,q) \int_{\mathbb{R}} \psi(q,x) U(x) \int_0^\infty V\lr{\frac{y}{X^2}} e\lr{\frac{yx}{q\mathcal{Q}}} (\log{y})^j \mathfrak{J}(m_1,m_2,x,q) dx dy.
\end{multline*}
As done earlier for $\mathcal{W}^\pm$, the $x$-integral gives us the upper bound $qX^\varepsilon/\mathcal{Q}$. If $m_1$ or $m_2$ is non-zero, we use the first derivative bound (see \autoref{exponential sum}) in $u$ or $v$ integral and treat the other one trivially to obtain the bound $R q^2X^\varepsilon/\mathcal{Q}^2$. Now, using the change of variable $y \rightarrow yX^2$, \autoref{CS C2 main}, and the above bound for integrals. We have the following estimate.
\begin{align} \label{Serr}
    \mathcal{S}_{err}(X,Y) &\ll \frac{X^{3+\varepsilon}Y}{\mathcal{Q}} \sum_{1 \leqslant q \leqslant \mathcal{Q}} \frac{1}{q^5} \mathop{\sum_{m_1 \ll RX^\varepsilon} \sum_{m_2 \ll RX^{1+\varepsilon}/Y}}_{m_1^2+m_2^2 \neq 0} q(q,Q^*(m_1,m_2)) \frac{R q^2}{\mathcal{Q}^2} \notag\\
    &\ll \frac{X^{3+\varepsilon}Y R}{\mathcal{Q}^3}  \mathop{\sum_{m_1 \ll RX^\varepsilon} \sum_{m_2 \ll RX^{1+\varepsilon}/Y}}_{m_1^2+m_2^2 \neq 0} \sum_{1 \leqslant q \leqslant \mathcal{Q}} \frac{(q,Q^*(m_1,m_2))}{q^2} \notag\\
    &\ll \frac{X^{3+\varepsilon}Y R}{\mathcal{Q}^3} \frac{R^2X}{Y} =X^{1+\varepsilon}R^3. 
\end{align}
\subsubsection{Main Term} 
Consider the expression for $\mathcal{S}_{main}$ from \autoref{S main}.
\begin{multline*}
    \mathcal{S}_{main}(X,Y) = \frac{X^3Y}{2 \mathcal{Q}} \sum_{1 \leqslant q \leqslant \mathcal{Q}} \frac{1}{q^5} \sum_{n \mid q} n d(n) \mathfrak{C}_1(0,0,n;q) \\
    \hspace{2cm}\times \sum_{j=0}^2 P_{2-j}(n,q)  \int_{\mathbb{R}} \psi(q,x) U(x) \int_0^\infty V\lr{y} (\log{yX^2})^j e\lr{\frac{xyX^2}{q\mathcal{Q}}} dy \\
    \times \mathop{\iint}_{\mathbb{R}^2} W_1(u) W_2(v) e\left(\frac{-x Q(uX,vY)}{q\mathcal{Q}} \right) du dv dx.
\end{multline*}
We can replace the weight function $V$ by $1$. To do this, let
\begin{equation*}
    \mathscr{V}^{M,j}(x) = \int_{1/2}^{5/2} (\log{yX^2})^j e\lr{\frac{xyX^2}{q\mathcal{Q}}} dy \quad \ \ j=0,1,2;
\end{equation*}
we need to estimate the remainder terms from 
\begin{equation*}
    \mathscr{V}^{\sharp,j}(x) = \Tilde{v}_x^{(j)}(1) - \mathscr{V}^{M,j}(x).
\end{equation*}
Write correspondingly 
\begin{equation*}
    \mathscr{S}_j^{\sharp}(X) = \mathscr{S}_j(X) - \mathscr{S}_j^M(X).
\end{equation*}
Thus, we have
\begin{equation} \label{S main into 3 sums}
    \mathcal{S}_{main}(X,Y) = \sum_{j=0}^2 \mathscr{S}_j(X).
\end{equation}
Notice that
\begin{equation*}
    \mathcal{V}^{\sharp,j}(x) = \int_{1/2}^{5/2} \lr{V\lr{y} -1} (\log{yX^2})^j e\lr{\frac{xyX^2}{q\mathcal{Q}}} dy \ll X^\varepsilon,
\end{equation*}
and using integration by parts we get that the integral is negligible unless $|x| \ll qX^\varepsilon/\mathcal{Q}$.
With this, we now consider
\begin{align} \label{S0sharp}
    \mathscr{S}_0^\sharp(X) &:= \frac{X^3Y}{2 \mathcal{Q}} \sum_{1 \leqslant q \leqslant \mathcal{Q}} \frac{1}{q^5} \sum_{n \mid q} n d(n) \mathfrak{C}_1(0,0,n;q) P_2(n,q)\notag\\
    &\hspace{1cm}\times \int_{\mathbb{R}} \psi(q,x) U(x) \mathscr{V}^{\sharp,j}(x) \mathop{\iint}_{\mathbb{R}^2} W_1(u) W_2(v) e\left(\frac{-x Q(uX,vY)}{q\mathcal{Q}} \right) du dv dx\notag \\
    &\ll  \frac{X^{3+\varepsilon}Y}{\mathcal{Q}}  \sum_{1 \leqslant q_2 \leqslant \mathcal{Q}} \frac{1}{q_2^5} \sum_{\substack{q_1 \leqslant \frac{\mathcal{Q}}{q_2} \\ q_1 \mid (2|\mathbf{A}|)^\infty}} \frac{1}{q_1^5} \sum_{n \mid q_1} n \frac{q^2 q_1}{n} \frac{R^2 q^2}{\mathcal{Q}^2} \notag\\
    &\ll \frac{X^{3+\varepsilon} Y R^2}{\mathcal{Q}^3} \sum_{q_2 \leqslant \mathcal{Q}} \frac{1}{q_2} \sum_{\substack{q_1 \leqslant \frac{\mathcal{Q}}{q_2} \\ q_1 \mid (2|\mathbf{A}|)^\infty}} 1  \ll X^\varepsilon Y R^2.
\end{align}
In the second last step above, we have used the first derivative bound for $u$ and $v$ integrals along with the restriction on $x$ integral.
Analogously, we have
\begin{equation} \label{S12sharp}
    \mathscr{S}_1^\sharp(X) \ll X^\varepsilon Y R^2 \quad \textrm{ and } \quad \mathscr{S}_2^\sharp(X) \ll X^\varepsilon Y R^2.
\end{equation}
Next, we have
\begin{align*}
    \mathscr{S}_j^M(X) &:= \frac{X^3Y}{2\mathcal{Q}} \sum_{1 \leqslant q \leqslant \mathcal{Q}} \frac{1}{q^5} \sum_{n \mid q} n d(n) \mathfrak{C}_1(0,0,n;q) P_{2-j}(n,q)\\
    &\hspace{1cm}\times \int_{\mathbb{R}} \psi(q,x) U(x) \mathscr{V}^{M,j}(x) \mathop{\iint}_{\mathbb{R}^2} W_1(u) W_2(v) e\left(\frac{-x Q(uX,vY)}{q\mathcal{Q}} \right) du dv dx \\
    &=\frac{X^3Y}{2\mathcal{Q}} \sum_{1 \leqslant q \leqslant \mathcal{Q}} \frac{1}{q^5} \ \sum_{n \mid q} n d(n) \mathfrak{C}_1(0,0,n;q) P_{2-j}(n,q) \\
    &\hspace{2cm}\times \int_{\mathbb{R}} \psi(q,x) U(x) \int_{1/2}^{5/2} (\log{yX^2})^j e\lr{\frac{xyX^2}{q\mathcal{Q}}} dy \\ 
    &\hspace{4cm}\times \mathop{\iint}_{\mathbb{R}^2} W_1(u) W_2(v) e\left(\frac{-x Q(uX,vY)}{q\mathcal{Q}} \right) du dv dx.
\end{align*}
Recall from \autoref{delta}, for small $q$, $\psi(q,x)$ can be replaced by $1$ at the cost of negligible error. We obtain
\begin{multline*}
    \mathscr{S}_j^M(X) = \frac{XY}{2\mathcal{Q}} \sum_{q=1}^\infty \frac{1}{q^5} \ \sum_{n \mid q} n d(n) \mathfrak{C}_1(0,0,n;q) P_{2-j}(n,q) \\
    \hspace{0.2cm}\times \int_{\mathbb{R}} U\lr{x} \int_{X^2/2}^{5X^2/2} (\log{y})^j e\lr{\frac{xy}{q\mathcal{Q}}} dy  \mathop{\iint}_{\mathbb{R}^2} W_1(u) W_2(v) e\left(\frac{-x Q(uX,vY)}{q\mathcal{Q}} \right) du dv dx,
\end{multline*}
up to some negligible error. Since 
\begin{align*}
    &\frac{XY}{2\mathcal{Q}} \sum_{ q \gg \mathcal{Q}^{1-2\varepsilon}} \frac{1}{q^5} \ \sum_{n \mid q} n d(n) \mathfrak{C}_1(0,0,n;q) P_{2-j}(n,q) \\
   &\hspace{0.2cm}\times \int_{\mathbb{R}} \psi(q,x) U\lr{x} \int_{X^2/2}^{5X^2/2} (\log{y})^j e\lr{\frac{xy}{q\mathcal{Q}}} dy \mathop{\iint}_{\mathbb{R}^2} W_1(u) W_2(v) e\left(\frac{-x Q(uX,vY)}{q\mathcal{Q}} \right) du dv dx \\
    &\ll \frac{X^{3+\varepsilon}Y}{\mathcal{Q}} \sum_{\substack{q_1 \gg \mathcal{Q}^{1-2\varepsilon} \\ q_1 \mid (2|\mathbf{A}|)^\infty}} \frac{1}{q_1^5} \sum_{n \mid q_1} n \ q_1^2 \sum_{q_2 \gg \frac{\mathcal{Q}^{1-2\varepsilon}}{q_1}} \frac{1}{q_2^5} q_2\phi(q_2) \frac{q_1q_2}{\mathcal{Q}} \\
   &\ll \frac{X^{3+\varepsilon}Y}{\mathcal{Q}^2} \sum_{\substack{q_1 \gg \mathcal{Q}^{1-2\varepsilon} \\ q_1 \mid (2|\mathbf{A}|)^\infty}} \frac{1}{q_1}  \sum_{q_2 \gg \frac{\mathcal{Q}^{1-2\varepsilon}}{q_1}} \frac{\phi(q_2)}{q_2^3} \\
   &\ll \frac{X^{3+\varepsilon}Y}{\mathcal{Q}^3} \ll X^\varepsilon Y.
\end{align*}
Make the change of variable $x \rightarrow xq/X$ and using
\begin{equation} \label{X removal from main}
    e\lr{\frac{-xQ(uX,vY)}{\mathcal{Q}^2}} = e\lr{-xQ(u,v)}e\lr{-x\lr{Bv^2 \lr{\frac{Y^2}{X^2}-1} +Cuv\lr{\frac{Y}{X}-1}}}, 
\end{equation}
we arrive at the following expression.
\begin{multline*}
    \mathscr{S}_j^M(X) = \frac{X^2Y}{2\mathcal{Q}} \sum_{q =1}^\infty \frac{1}{q^4} \ \sum_{n \mid q} n d(n) \mathfrak{C}_1(0,0,n;q) P_{2-j}(n,q) \\
    \times \int_{\mathbb{R}} \int_{1/2}^{5/2} (\log{yX^2})^j e\lr{xy} dy  \mathop{\iint}_{\mathbb{R}^2} W_1(u) W_2(v) e\left(-xQ(u,v) \right) du dv dx +\mathcal{O}_c(X^{-c})
\end{multline*}
for any $c>1$. Here, we have used the fact that the second factor on the right side of \autoref{X removal from main} has no oscillations and hence it can be absorbed to the smooth weight function. Furthermore, the weight function can be replaced by $1$ as before. Using the properties of the logarithm function, we can further simplify the main term. Therefore, for any $c>1$, we have
\begin{multline} \label{SMj}
   \mathscr{S}_j^M(X) = \frac{X^2Y}{2\mathcal{Q}} \sum_{q =1}^\infty \frac{1}{q^4} \ \sum_{n \mid q} n d(n) \mathfrak{C}_1(0,0,n;q) P_{2-j}(n,q) \\
   \times \sum_{k=0}^j \binom{j}{k} 2^{j-k} (\log{X})^{j-k} \mathcal{J}_k
   +\mathcal{O}_c(X^{-c}),
\end{multline}
where
\begin{equation} \label{Jk}
    \mathcal{J}_k := \int_{\mathbb{R}} \int_{1/2}^{5/2} (\log{y})^k e\lr{xy} dy  \mathop{\iint}_{\mathbb{R}^2} W_1(u) W_2(v) e\left(-xQ(u,v) \right) du dv dx.
\end{equation}
Combining the results from equations \autoref{S main into 3 sums}, \autoref{S0sharp}, \autoref{S12sharp}, and \autoref{SMj}, we obtain the following lemma.

\medskip

\begin{lemma} \label{lemma main term}
    For any $c>1$, the main term of \autoref{d3 thm} is given as
\begin{multline*}
    \mathcal{S}_{main}(X,Y) = 2XY(\log{X})^2 \mathcal{C}_0 \mathcal{J}_0 + XY \log{X} (\mathcal{C}_1 \mathcal{J}_0 + 2 \mathcal{C}_0 \mathcal{J}_1) \\+ \frac{1}{2}XY (\mathcal{C}_2\mathcal{J}_0 + \mathcal{C}_1\mathcal{J}_1 + \mathcal{C}_0 \mathcal{J}_2) +\mathcal{O}_c(X^{-c}),
\end{multline*}
where 
\begin{multline*}
    \mathcal{C}_j := \sum_{q =1}^\infty \frac{1}{q^4} \sum_{n \mid q} n d(n) P_j(n,q) \starsum_{a \textrm{ mod } q} \sum_{\alpha_1 \textrm{mod } q} \sum_{\alpha_2 \textrm{mod } q} e\lr{\frac{-aQ(\alpha_1,\alpha_2)}{q}} \\ \times S\lr{\ol{a},0;\frac{q}{n}}, 
\end{multline*}
and $\mathcal{J}_j's$ are integrals defined in \autoref{Jk}.
\end{lemma}

\medskip

\noindent Finally, we have proved the asymptotic formula stated in  \autoref{d3 thm} after combining the results of the above lemma and \autoref{Serr}, together with the bound in \autoref{Bound S}.

\medskip
\noindent{\bf Acknowledgment:} We express our gratitude to the Department of Mathematics and Statistics at the Indian Institute of Technology Kanpur, India, for providing a conducive research environment. H. Chanana is grateful for the support received through the University Grants Commission, Government of India (UGC-JRF/SRF). The authors would like to thank the anonymous referees for their careful reading and valuable comments, which improved the presentation of the paper. 

\medskip
\printbibliography

\end{document}